\documentclass[12pt]{article}
\usepackage{verbatim}
\usepackage{amssymb,amsmath,amsthm,bm}
\usepackage{graphicx}
\usepackage[all]{xy}
\usepackage{bbm}
\usepackage{epsfig}
\newtheorem{theorem}{Theorem}[section]
\newtheorem{corollary}[theorem]{Corollary}
\newtheorem{lemma}[theorem]{Lemma}
\newtheorem{proposition}[theorem]{Proposition}
\newtheorem{conjecture}[theorem]{Conjecture}

\newcommand{\Prob} {{\mathbb P}}
\newcommand{\Z}{{\mathbb Z}} 
\newcommand{\E}{{\mathbb E}}

\newcommand{\R}{{\mathbb{R}}}
\newcommand{\C}{{\mathbb C}}

\newcommand{\dist}{{\rm dist}}
\newcommand{\x}{{\bf x}}
\newcommand{\y}{{\bf y}}

\usepackage{amsthm}

\makeatletter
\newtheorem*{rep@theorem}{\rep@title}
\newcommand{\newreptheorem}[2]{%
	\newenvironment{rep#1}[1]{%
		\def\rep@title{#2 \ref{##1}}%
		\begin{rep@theorem}}%
		{\end{rep@theorem}}}
\makeatother

\newreptheorem{theorem}{Theorem}
\newreptheorem{corollary}{Corollary}

\def \Im {{\rm Im}}

\def \p {\partial}

\def \Half {{\mathbb H}}

\def \Disk {{\mathbb D}}

\def \eset {\emptyset}
\def \C {{\mathbb C}}
\def \nat {{\mathbb N}}
\def \BMtrans {{\tilde p}}
\def\bx {{\mathbf x}}
\def \by {{\mathbf y}}
\def \const {{q}}

\def \saws {{\cal R}}
\def \paths {{\cal K}}

\def \linehere { {\hrule}}
\def \labove { \mtwo \linehere \linehere \linehere \ms   }

\def \lbelow {{\ms \linehere \linehere \linehere \mtwo}}

\def \mtwo {{\medskip \medskip}}
\def \ms {{\medskip}}
\newcommand {{\whoknows}} {{\mathcal C}}

\newenvironment{remark}[1][Remark]{\begin{trivlist}
\item[\hskip \labelsep {\bfseries #1}]}{\end{trivlist}}
\newenvironment{definition}[1][Definition]{\begin{trivlist}
\item[\hskip \labelsep {\bfseries #1}]}{\end{trivlist}}

\newcommand \hcap  {{\rm hcap}}

\newcommand \slepart  {{\Psi}}
\newcommand \stoptime {\tau}

\newenvironment{advanced}
{ \labove \begin{quote} \begin{small}}
{ \end{small}\end{quote} \lbelow }

\def \begad{\begin{advanced}}
\def \endad{\end{advanced}}

\def \paths{{\mathcal K}}
\def \saws {{\mathcal W}}
\def \z {{\bf z}}

\def \bgamma {\boldsymbol \gamma}
\def \btheta
{ {\boldsymbol \theta}}
\def  \bz  {{\bf z}}

\newcommand{{\pe}}  {\partial_e}
\newcommand {{\lodd}} {{\mathcal J}}
\newcommand{{\inrad}} {{\rm inrad}}

\newcommand {{\cent}} {{\bf c}}
\newcommand {{\eb}}  {{\bf e}}

\newcommand{{\saps}}  {{\mathcal X}}
\newcommand {{\dyadic}}  {{\mathcal Q}}
\newcommand {{\curves}} {\paths}
\newcommand {{\ball}} {{\mathcal B}}
\newcommand {{\measures}} {{\mathcal M}}
\newcommand{{\brown}}  {{\nu}}
\newcommand{{\osc}}  {{\rm osc}}
\newcommand  {{\rbrown}} {{\mu}}
\newcommand{{\lcurves}}{{\curves_L}}
\newcommand{{\bcurves} } {\curves^{\rm bub}}
\newcommand {{\bbrown}}{{\nu^{\rm bub}}}
\newcommand {{\n}} {{\bf n}}

\newcommand \bt {{\bf t}}

\newcommand  \bfeta {{\bm \eta}}

\newcommand \Ecal  {{\mathcal E}}
\newcommand {{\loopterm}}  {{\mathcal L}}
\newcommand {\pastloopterm} {\hat {\mathcal L}}
\newcommand {\pastloops}[2] { \hat {L}^{#1}_{#2}}
\newcommand {\loops}[2]{L^{#1}_{#2}}
\newcommand {\sumcsc} {\psi}
\newcommand {\indicator}{I}
\newcommand {\pastindicator}{\hat I}
\newcommand {\Integral}[1] {\mathcal I_{#1}}

\newcommand  \X  {{\bf X}}
\newcommand \B  {{\bf B}}
\newcommand\bdry {{\partial}}

\def \Ecal   {{\mathcal E}}

\def \multE {{\Ecal}}
\def \multProb {{\mathcal P}}
\def \BMProb {{P}}

\newcommand{\tildeN}[2]{\tilde N_{#1, #2}}

\newcommand{\abs}[1]{\left\vert #1 \right\vert}
\newcommand{\nin}{\not\in}

\newcommand{\drjradial}[2]{\tilde \theta^{#1}_{#2}}
\newcommand{\drjcommon}[2]{\xi ^{#1}_{#2}}

\begin{document}
	\title{N-sided Radial Schramm-Loewner Evolution}
	\author{Vivian Olsiewski Healey\footnote{research supported  in part by NSF DMS-1246999.} \,and Gregory F. Lawler\footnote{research suported  by NSF DMS-1513036}}
	\date{\today}
	\maketitle 
	
\abstract{
We use the interpretation of the Schramm-Loewner evolution
as a limit of path measures tilted by a loop term in order
to motivate the definition of $n$-radial SLE going to
a particular point.  In order to justify the
definition we prove that the measure obtained by an
appropriately normalized loop term on $n$-tuples
of paths has a limit.  The limit  measure can be 
described as $n$ paths moving by the Loewner equation
with a driving term of Dyson Brownian motion.  While the
limit process has been considered before, this paper shows why it naturally arises as a limit of configurational measures obtained
from loop measures.}

\section{Introduction}

Multiple Schramm-Loewner evolution has been
studied by a number of authors   including \cite{KL}, \cite{Dubedat}, \cite{Mohammad}, \cite{PWGlobal_Local_SLE}, \cite{BPW_GlobalSLE} (chordal) and \cite{Z2SLEbdry}, \cite{Z2SLEint} ($2$-sided radial). For $\kappa\leq 4$, domain $D$, and $n$-tuples $\boldsymbol x$ and $\boldsymbol y$ of boundary points,  multiple chordal $SLE$ from $\boldsymbol x$ to  $\boldsymbol y$ in $D$ is defined as the measure absolutely continous with the $n$-fold product measure of chordal $SLE$ in $D$ with Radon-Nikodym derivative
\begin{equation}\label{chordal_tilt}
Y(\bgamma)= I(\bgamma) \exp\left\{\frac \cent 2 
\sum_{j=2}^n
\, m[K_j(\bgamma)] \right\},
\end{equation}
where $I(\bgamma)$ is the indicator function of
\[
\{ \gamma^j\cap \gamma^k =\emptyset, \, 1\leq j <k \leq n \},
\]
and $m[K_j(\bgamma)]$ is the Brownian loop measure of loops that intersect at least $j$ paths (see, e.g.,  \cite{JL_partition_function} for this result; see \cite{ConfInvLERW} for the construction of Brownian loop measure). 
We would like to define multiple radial $SLE$ by direct analogy with the chordal case, but this is not possible for two reasons. First, in the radial case the event $I(\bgamma)$ would have measure $0$, and second, the Brownian loop measure $m[K_j(\bgamma)]$ would be infinite, since all paths approach $0$. Instead, the method will be to construct a measure on $n$ paths that is absolutely continuous with respect to the product measure on $n$ independent radial $SLE$ curves with Radon-Nikodym derivative analogous to (\ref{chordal_tilt}) but for both $I(\bgamma)$ and $m[K_j(\bgamma)]$ depending only on the truncations of the curves at a large time $T$. Taking $T$ to infinity then gives the definition of multiple radial $SLE$. The precise details of this construction, the effect on the driving functions, and the rate of convergence of the partition function are the main concern of this work.

Schramm-Loewner evolution, originally introduced in \cite{Schramm}, is a distribution on a curve in a domain $D\subset \C$ from a boundary point to either another boundary point (\emph{chordal} $SLE$) or an interior point (\emph{radial} $SLE$). In both the chordal and radial cases, there are various  ways to define $SLE$ measure.   Schramm's original observation
was that any probability measure on curves satisfying conformal
invaraiance and the domain Markov property can be
described  in the upper half plane or the disc using the Loewner
differential equation.  More precisely, after a suitable time
change,  it is the measure on parameterized curves $\gamma$ such that for each $t\in [0, \infty)$, $D=g_t\left(D\setminus \gamma[0,t]\right)$, where $g_t$ solves the Loewner equation:
\[ \text{Chordal: }  \dot g_t(z) =\frac{a}{g_t(z)-B_t} ,\quad g_0(z)=z \]
 \[ \text{Radial: }  \dot g_t(w) = 2a g_t(w) \frac{z_t + g_t(w) }{z_t-g_t(z)} ,\quad g_0(z)=z, \]
 where $a=2/\kappa$, $B_t$ is a standard Brownian motion, and $z_t= e^{2i B_t}$.
However, this dynamical interpretation is somewhat artificial in the
sense that the curves typically arise from limits of models
in equilibrium physics and are not ``created'' dynamically using
this equation.  Indeed, the dynamic interpretation is just a way
of describing conditional distributions given certain amounts
of information.   When studying $SLE$, one goes back and
forth between such dynamical interpretations and configurational
or ``global'' descriptions of the curve.

One aspect of the global pespective is that 
%
%
%
radial $SLE$ measure in different domains may be compared by also considering the partition function $\slepart_D(z,w)$, which assigns a total mass to the set of $SLE$ curves from $z$ to $w$ in the domain $D$. It is defined as the function with normalization $\slepart_\Disk(1,0)=1$ satisfying conformal covariance:
\begin{equation}\label{eq:slepart_conformalcovariance}
\slepart_D(z,w)= \abs{f'(z)}^b \abs{f'(w)}^{\tilde b}\slepart_{D'}(z',w'),
\end{equation}
where $f(D)=D'$, $f(z)=z'$, $f(w)=w'$, and
\[
b=\frac{6-\kappa}{2\kappa}=\frac{3a-1}{2}, \qquad
 \tilde b = b \,\frac{\kappa-2}{4}=b\,\frac{1-a}{2a}
\]
are the boundary and interior scaling exponents. (This definition requires sufficient smoothness of the boundary near $z$.) Another convention defines the partition function with an additional term for the determinant of the Laplacian, however, the benefit of our convention is that value of the partition function is equal to the total mass.

 Considering $SLE$ as a measure with total mass allows for direct comparison between $SLE$ measure in $D$ with $SLE$ measure in a smaller domain $D'\subset D$. This comparison is called either \emph{boundary perturbation} or the \emph{restriction property}, and is stated precisely in Proposition \ref{restriction} \cite{Mohammad}.

Multiple chordal $SLE$ was first considered in \cite{BBK,   Dubedat,KL}. Dub\'edat  \cite{Dubedat} shows that two (or more) $SLE$s commute only if a system of differential equations is satisfied, and the construction holds until the curves intersect. Using this framework, the uniqueness of global multiple $SLE$ is shown in \cite{PK} \cite{PWGlobal_Local_SLE} and \cite{BPW_GlobalSLE}. In these works, the term \emph{local} $SLE$ is used to refer to solutions to the Loewner equation up to a stopping time, while \emph{global} $SLE$ refers to the measure on entire paths.

\begin{figure}
	\begin{center}
		\includegraphics[scale=1.2]{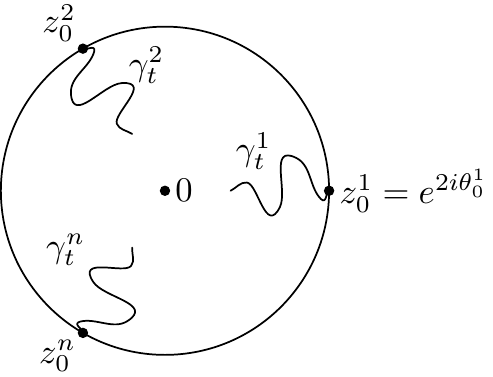}
		\caption{Initial segments of $n$-radial $SLE$}
		\label{fig:unit_circle}
	\end{center}
\end{figure}

This work builds on the approach of \cite{KL}, which relies on the loop interpretation to give a global definition for $0<\kappa\leq 4$.
However, because we have to take limits, we will need to use
both global and dynamical expressions.  The
dynamical description relies  on computations concerning the radial Bessel process (Dyson Brownian motion on the circle)  and go back
to  \cite{Cardy}, and hold in the more general setting of $\kappa<8$.


\begin{figure}
	\begin{center}
		\includegraphics[scale=.22]{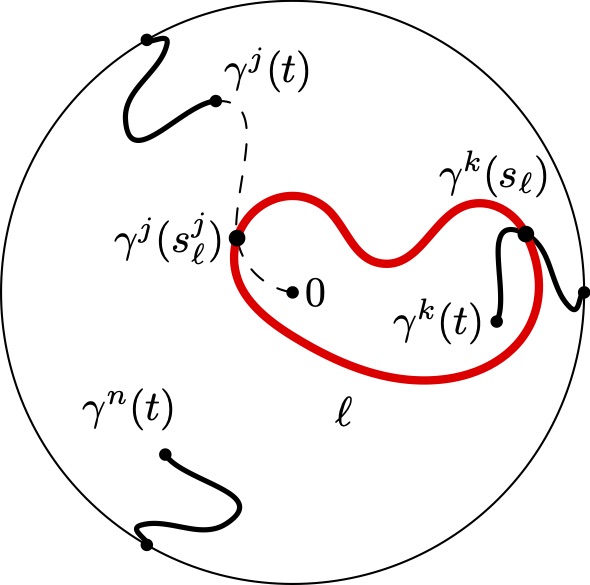} \hspace{3em}
		\includegraphics[scale=.22]{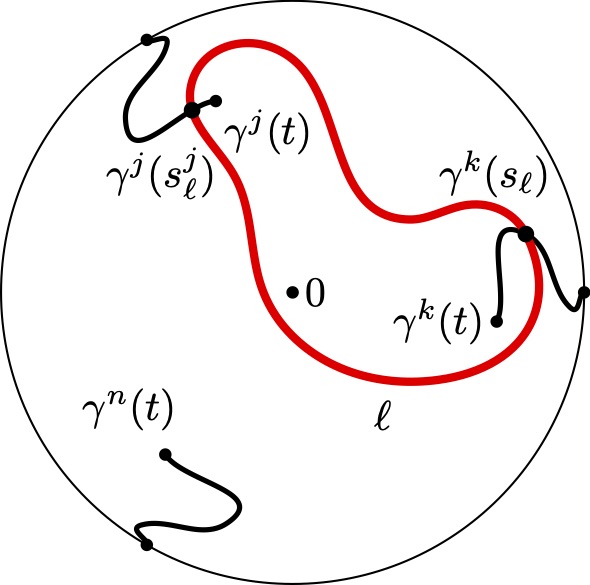}
		\caption{Left: a loop $\ell$ contained in $\loops{j}{t}$. Right: a loop $\ell$ contained in $\pastloops{j}{t}$. The key difference is that loops in $\pastloops{j}{t}$ must intersect $\gamma^j$ before time $t$. In both cases, $\ell$ intersects $\gamma^k_t$ before $\gamma^j$, i.e. $s(\ell)<s^j(\ell)$. }
		\label{fig:loops}
	\end{center}
\end{figure}

Our main result is the following. Let $n$ be a positive integer
and $\bgamma=(\gamma^1, \ldots, \gamma^n)$  an $n$-tuple of curves from $z^j_0 \in \bdry \Disk$ to $0$ with driving functions $z^j_t=e^{2i\theta^j_t}$. We will assume that the curves are parameterized using the $a$-common parameterization, which is defined in \S \ref{sec:locindep}. Let $\Prob$ denote the $n$-fold product measure on independent radial $SLE$ curves from $\gamma^j(0)$ to $0$ in $\Disk$ with this parameterization. (See Figure~\ref{fig:unit_circle}.)
Let $\loops{j}{t} = \loops{j}{t}(\bgamma_t)$ be the  set of loops $\ell$ that hit the curve $\gamma^j$ and at least one initial segment $\gamma^k_t$ for $k=1, \ldots, n$, $k\neq j$ but do not hit $\gamma^j$ first. (See the lefthand side of Figure \ref{fig:loops}.) Here we are measuring the ``time'' on the
curves $\bgamma$ and not on the loops.  Define
\[   \loopterm_t = \indicator_t \, \exp \left\{ \frac{\cent}{2}
\sum_{j=1}^n m_\Disk(\loops{j}{t}) \right\} \] 
where  $I_t$ is the indicator function that
$\gamma^j_t \cap \gamma^k = \eset $
for $j \leq k $, and $m_\Disk$ is the Brownian loop measure.

\begin{reptheorem}{maintheorem} 
	Suppose  $0 < \kappa \leq 4$ and  
 $t>0$.  For each $T>t$, let 
	$\mu_{T}=\mu_{T,t}$ denote the measure
	whose Radon-Nikodym
	derivative with respect to $\Prob$ is
	\[                  \frac{  \loopterm_T}
	{ \E^{\btheta_0}\left[  \loopterm_T\right]}.\]
	Then as $T \rightarrow \infty$, the measure $\mu_{T,t}$,
	viewed as a measure on curves stopped at time $t$, 
	approaches  a probability measure   with respect to the variation distance. 
	
	Moreover, the measures  
	are consistent and give a probability measure 
	on curves $\{\bgamma(t): t \geq 0\}$.   This measure can be decribed as the solution
	to the $n$-point Loewner equation with
 driving functions $z^j_t=e^{2i \theta^j_t}$ satisfying
	\begin{equation}
	d \theta_t^j = 2a \sum_{k \neq j}
	\cot(\theta_t^j -\theta_t^k) \, dt + dW_t^j,
	\end{equation}
	where $W^j_t$ are independent standard Brownian motions.
\end{reptheorem}

A key step in the proof  is Theorem \ref{exponentialrateofconv1} 
which gives exponential convergence of 
a particular partition function for $n$-radial Brownian motion.
This theorem is valid for $0 < \kappa < 8$, but only in the
$\kappa \leq 4$ case can we apply this to our model and give
a corollary that we now describe.
  Let $ \mathcal X  ={ \mathcal X}_n$ denote the 
set  of ordered pairs $\btheta = (\theta^1,\ldots,\theta^n)$ in the torus
$[0,\pi)^n$ for which there are representatives with
$0 \leq \theta^1 < \theta^2 < \ldots < \theta^n < \theta^1 + \pi$.
Denote
 \[ F_a(\btheta) =  \prod_{1\leq j<k\leq n} |\sin(\theta^k - \theta^j)|^a
\]
\[
\Integral{a} =
\int_{\mathcal X} F_a(\btheta) \, d \btheta.
\]
	\[   \beta = \beta(a,n)  = \frac{a(n^2-1)}{4}.\]
	
\begin{repcorollary}{expconvloopversion}
	If $a \geq 1/2$, there exists $u =u(2a,n)> 0$ such that 
	\[ \E^{\btheta_0}\left[ \loopterm_t \right] =  e^{-2an\beta t} \, \frac{\Integral {3a}}{\Integral{4a}} F_a(\btheta)
	[ 1 + O(e^{-ut})].  \]
\end{repcorollary}
%

The paper is organized as follows. Section \ref{sec:discrete} describes the multiple $\lambda$-SAW model, a discrete model which provides motivation and intuition for the perspective we take in the construction of $n$-radial $SLE$. Section \ref{sec:preliminaries} gives an overview of the necessary background for the radial Loewner equation. Section \ref{sec:MultipleSLE} contains the construction of $n$-radial $SLE$ (Theorem \ref{maintheorem}) as well as locally independent $SLE$. The necessary results about the $n$-radial Bessel process are stated here in the context of $\kappa\leq 4$ without proof. Finally, section \ref{sec:DysonBM} contains our results about the $n$-radial Bessel process, including Theorem \ref{exponentialrateofconv1}. These results hold for all $\kappa<8$ and include proofs of the statements that were needed in section \ref{sec:MultipleSLE}.

\section{Preliminaries}\label{sec:preliminaries}

\subsection{Discrete Model}\label{sec:discrete}

Although we will not prove any results about convergence
of a discrete model to the continuous, much of the motivation
for our work comes from a belief that $SLE$
is a scaling limit of the ``$\lambda$-SAW'' described first
in \cite{KL}. In particular, the key insight needed to prove Theorem \ref{maintheorem}, the use of the intermediate process \emph{locally independent} $SLE_\kappa$ as a step between independent $SLE_\kappa$ and $n$-radial $SLE_\kappa$, was originally formulated by considering the partition function of multiple $\lambda$-SAW paths approaching the same point. For this reason, we describe the discrete model in detail here.


The model weights self-avoiding paths using the \emph{random walk loop measure}, so we begin by defining this.
A (rooted) random walk loop in $\Z^2$ is a nearest
neighbor path $\ell = [\ell_0,\ell_1,\ldots,\ell_{2k}]$
with $\ell_0 = \ell_{2k}$.  The loop measure gives measure
$\hat m(\ell) = (2k)^{-1} \, 4^{-2k}$ to each nontrivial loop of length $2k >0$.  If $V \subset A \subset \Z^2$, we let
\[       F_V(A) = \exp \left\{\sum_{\ell \subset A,
   \ell \cap V \neq \eset  } m(\ell)   \right\}, \]
   that is, $\log F_V(A)$ is the measure of loops
   in $A$ that intersect $V$. 

We fix $n$ and some $r_n > 0$ such 
that there exists $n$ infinite self-avoiding paths starting
at the origin that have no intersection after they first
leave the ball of radius $r_n$.
 (For $n \leq  4$, we can choose $r_n = 0$ but for larger
 $n$ we need to choose $r_n$ bigger because one cannot
 have five nonintersecting paths starting at the origin.
 This is a minor discrete detail that we will not worry about.)  If $A \subset \Z^2$ is a finite, simply connected
 set containing the disk of radius $r_n$ about the origin, we let $\saws_A$ denote the set of
 self-avoiding walks $\eta$ starting at $\p A$,
 ending at $0$, and otherwise staying in $A$.
 As a slight abuse of notation, we will write $\eta^1
 \cap \eta^2 = \eset$ if the paths have no intersections
 other then the beginning of the reversed paths up to
 the first exit from the ball of
 radius $r_n$.    (If $n \leq 4$ and $r_n = 0$, this means
 that the paths do not intersect anywhere except their terminal
 point which is the origin.)

 If $\bfeta = (\eta^1,\ldots,\eta^n)$ is an $n$-tuple
 of such paths, we let $I(\bfeta)$ be the indicator function
 of the event that $\eta^j \cap \eta^k = \eset$ for all $j\neq k$.  We write $|\eta^j|$ for the number of edges in $\eta^j$
 and $|\bfeta| = |\eta^1| + \cdots + |\eta^j|$.
  Let $\bar \saws_A = \bar \saws_{A,n} $ denote the set of $n$-tuples $\bfeta$
 in $\saws_A$ with $I(\bfeta) = 1$.  We then consider
 the measure on configurations given by
 \[      \nu_{A,\cent}
 (\bfeta) =  \exp\{- \beta |\bfeta|\}
  \, I(\bfeta)\,  F_\bfeta(A)^{\cent/2} . \]
Here $\beta = \beta_\cent$ is a critical value under which
the measure becomes critical.  If $\z \in (\p A)^n$, we write
$\bar \saws_A(\z)$ for the set of $\bfeta \in \bar \saws_A$
such that $\eta^j$ starts at $z^j$.

Suppose $D$ is a bounded, simply connected domain in
$\C$ containing the origin and let $\z = (z^1,\ldots,z^n)$
be an $n$-tuple of distinct points in $\p D$ oriented
counterclockwise.  For ease, we assume that for
each $j$, $\p D$ in a neighborhood of $z^j$ is a straight
line segment parallel to the coordinate axes (e.g., $D$
could be a rectangle and none of the $z^j$ are corner
points).  For each lattice spacing $N^{-1}$, let
$A_N$ be an approximation of $N D$ in $  \Z^2$ and
let $\z_N = (z^1_N,\ldots,z^n_N)$ be lattice
points corresponding to $N \z$.  We can consider
the limit as $N \rightarrow \infty$ of the measure
on scaled configurations $N^{-1} \, \bfeta$ given by
$\nu_{A_N,\cent}$ restricted to $\bar \saws_{A_N}(\z_N)$.

\begin{conjecture}
Suppose $ \cent \leq 1$.  Then there exist $b,\tilde b_n$
and, critical
$\beta = \beta_\cent$ and a partition function $\slepart^*(D;\z,0)$
such that as $N \rightarrow \infty$,
\[ \nu_{A_N,\cent} (\bar \saws_{A_N}(\z_N))
    \sim    \slepart^*(D;\z,0)  \, N^{nb}
     \, N^{\tilde b_n} . \]
Moreover, the scaling limit $N^{-nb} \, N^{-\tilde b_n}
  \,   \nu_{A_N,\cent} $ is $n$-radial $SLE_\kappa$, $\mu_D(\z,0)$ with
  partition function $\slepart(D;\z,0).$  If $f:D \rightarrow
  f(D)$ is a conformal transformation with $f(0) = 0$, then
  \[    f \circ \mu_D(\z,0) = |f'(\z)|^b \, |f'(0)|^{\tilde{b}_n }\, \mu_{f(D)}(f(\z),0).\]  Here
  $f(\z)  = (f(z^1),\ldots,f(z^n))$ and $ f'(\z)
  = f'(z^1) \cdots f'(z^n)$. 
\end{conjecture}

This conjecture is not precise, and since we are not planning
on proving it, we will not make it more precise.  The main goal of this paper is to show that assuming the conjecture informs us as to what $n$-radial $SLE_\kappa$ should be and what the exponents $b, \tilde b_n$ are.

The case $n=1$ is usual radial $SLE_\kappa$ for $\kappa \leq 4$ and the relation is
\[   b = \frac{6-\kappa}{2\kappa}, \;\;\;\;
  \tilde b_1 = \tilde b = b \, \frac{\kappa - 2}{4} , \;\;\;\;
    \cent =  \frac{(3 \kappa-8)(6-\kappa)}{2\kappa}.\]
This is understood rigorously in the case of  $\cent = -2, \kappa = 2$
since the model is equivalent to the loop-erased random walk.
For other cases it is an open problem.  For $\cent = 0$,
  it is essentially equivalent to most of the very
hard open problems about self-avoiding walk.  However, assuming the conjecture and using the fact that the limit should satisfy
the restriction property, one can determine $\kappa = 8/3,
 \cent = 0, b = 5/8, \tilde b = 5/48$.  The critical exponents
 for SAW can be determined (exactly but
 nonrigorously) from these values.

The case $n=2$ is related to {\em two-sided radial $SLE_\kappa$} which can also be viewed as chordal $SLE_\kappa$ from $z^1$ to $z^2$, restricted to paths that go through the
origin.  In this case, $b_2 = d-2$ where $d =   1 + \frac \kappa 8$ is the fractal dimension of the paths.

\subsection{Radial $SLE$ and the restriction
property}  \label{radialsec}

   The radial Schramm-Loewner evolution  with
parameter $\kappa =2/a$ ($SLE_\kappa$) from $ z = e^{2i\theta}$
to the origin in the unit disk is defined as the random
curve $\gamma(t)$ with the following properties.
Let $D_t$ be the component of $\Disk \setminus \gamma[0,t]$
containing the origin. If $g_t: D_t \rightarrow \Disk
$ is the conformal transformation with $g_t(0) = 0,
g_t'(0) > 0$, then $g_t$ satisfies
\[   \dot g_t(w) = 2a \, g_t(z) \, \frac{e^{2iB_t}+ g_t(w)}
    {e^{2iB_t} - g_t(w)} , \;\;\;\; g_0(w) = w, \]
where $B_t$ is a standard Brownian motion.
More precisely, this is the definition of radial
$SLE_\kappa$ when the curve has been parameterized so that $g_t'(0) = e^{2at}$.

We will view $SLE_\kappa$ as a measure on curves modulo reparameterization (there is a natural parameterization that can be given to the curves, but we
will not need this in this paper).   
We extend $SLE_\kappa$ to be a probability measure
$\mu^\#_D(z,w)$  where $D$ is a simply connected
domain, $z \in \p D$ and $w \in D$ by conformal
transformation.  It is also useful for us to consider
the non-probability measure $\mu_D(z,w) =
 \slepart_D(z,w)\, \mu_D^\#(z,w)$. Here
 $\slepart_D(z,w)$ is the radial partition function
 that can be defined by $\slepart_\Disk(1,0) = 1$
 and the scaling rule
\[    \slepart_D(z,w) = |f'(z)|^b \, |f'(w)|^{\tilde b}
 \, \slepart_{f(D)}(f(z),f(w)), \]
 where
 \[   b=  \frac{6 - \kappa}{2\kappa}
 = \frac{3a-1}{2} , \;\;\;\; \tilde b = b \, \frac{\kappa-2}{4} =
     b \, \frac{1-a}{2a} , \]
  are the boundary and interior scaling exponents.
  This definition requires sufficient smoothness of
  the boundary near $z$.  However,  if $D' \subset D$
  agree in neighborhoods of $z$, then the ratio
  \[   \slepart(D,D';z,w):=
      \frac{\slepart_{D'}(z,w)}{\slepart_D(z,w)} \]
   is a conformal invariant and hence is well defined
   even for rough boundary points.

 We will need the restriction property for radial
 $SLE_\kappa, \kappa \leq 4$.  We state here it in a way that does
 not depend on the parameterization, which is the form that we will use.
 
 \begin{definition} 
 If $D$ is a domain and $K_1,K_2$ are disjoint subsets
 of $D$, then $m_D(K_1,K_2)$ is the Brownian loop
 measure of loops that intersect both $K_1$ and $K_2$.
 \end{definition}
 
 \begin{proposition}[Restriction property]
 Suppose $\kappa \leq 4$ and $D = \Disk \setminus K$
 is a simply connected domain containing the origin.
 Let $z \in \p D$ with $\dist(z,K) > 0$, and let $\gamma$
 be a radial $SLE_\kappa$ path from $z$ to $0$
 in $\Disk$.  Let
 \[     M_t = 1\{\gamma_t \cap K = \eset\}  \, \exp\left\{\frac \cent 2
 \, m_\Disk(\gamma_t,K) \right\} \, \slepart_t , \]
 where $\slepart_t = \slepart(\Disk \setminus
 \gamma_t, D \setminus \gamma_t; \gamma(t),0).$
 Then $M_t$ is a uniformly integrable martingale,
 and the probability
 measure obtained from Girsanov's theorem by tilting by $M_t$ is $SLE_\kappa$ from $z$ to $0$ in
 $D$.  In particular,
 \begin{equation} \label{restriction}
 \E\left[1\{\gamma \cap K = \eset\}
  \, \exp\left\{\frac \cent 2
 \, m_\Disk(\gamma,K) \right\} \right] = \E[M_\infty]
  = M_0 = \slepart(\Disk,D; z,0).
  \end{equation}
 \end{proposition}

See \cite{Mohammad} for a proof.  It will be useful for us to discuss the
ideas in the proof.  We parameterize the curve as above and we consider $\Psi_t$, the ratio of partition functions at time $t$ of $SLE$ in $D \setminus \gamma_t$ with
$SLE$ in $\Disk \setminus \gamma_t$.  Using the scaling rule and It\^o's formula, one  computes the SDE for
$\slepart_t,$
\[    d\slepart_t = \slepart_t \, \left[A_t \, dt + R_t
 \, dB_t\right].\]
This is not a local martingale, so we find the compensator and let
\[        M_t = \Psi_t \, \exp \left\{-\int_0^t A_s \, ds \right\},\]
which satisfies
\[            dM_t = R_t \, M_t \, dB_t.\]
This is clearly a local martingale.  The following
observations are made in the calculations:
\begin{itemize}
\item   The compensator term is the same as 
$\exp\left\{\frac \cent 2
 \, m_D(\gamma,K) \right\} $.
 \item  If we use Girsanov theorem and tilt by this
 martingale, we get the same distribution on paths as
 $SLE_\kappa$ in $D$.  The latter distribution 
 was defined by conformal invariance. 
 \end{itemize}
All of this is valid for all $\kappa$ up to the first time
$t$ that $\gamma(t) \in K$.  For $\kappa \leq 4$,
we now use the fact that radial $SLE$ in $D$ never
hits $K$ and is continuous at the origin.  This allows
us to conclude that it is a uniformly integrable martingale.  With probability one in the new measure
we have $\gamma \cap K = \eset$ and hence
we can conclude the proposition.

We sketched the proof in order to see what happens
when we allow the set $D$ to shrink with time. In particular, let $D_t = \Disk \setminus K_t$, where $K_t$ grows with time, and let
\[        \slepart_t = \slepart(\Disk \setminus \gamma_t,
 D_t\setminus \gamma_t;\gamma(t),0), \]
 \[   T = \inf\{t: \gamma_t \cap K_t \neq \eset \}.\]
For $t < T$, we can again consider $SLE$ tilted
by $\slepart_t$.  However, since $K_t$ is growing, the loop term is more subtle in this case. Roughly speaking, the relevant loops are those that intersect $K_s$ for some $s$ smaller than their first intersection time with $\gamma_t$.
%
%
More precisely, the local martingale
has the form
\[                  \loopterm_t^{\cent/2} \, \exp\left\{-
\int_0^t A_s \, ds \right\} \, \Psi_t, \]
where
\begin{itemize}
\item $\log \loopterm_t$ is the Brownian measure 
of loops $\ell$ that hit $\gamma_t$ and satisfy the following: if $s_\ell$ is the smallest time with $\gamma(s_\ell)\in \ell$, then $l\cap K_{s_\ell} \ne \eset$.

\item 
$    A_t = \rho'(0)$, for
\[\rho(\epsilon)= \rho_t(\epsilon)  =  \slepart(\Disk \setminus \gamma_t,
  D_{t+\epsilon}\setminus \gamma_t; \gamma(t),0).\]
We assume that $A_t$ is well defined, that is, that $\rho$ is differentiable.

\end{itemize}
When we tilt by $\slepart_t$, the process at time $t$ moves like $SLE$ in $D_t$.  We will only consider this up to time $T$.  

\section{Measures on $n$-tuples of paths}  \label{sec:MultipleSLE}

We will use a similar method to define two measures on
$n$-tuples of paths which can be viewed
as process taking values in    $\overline \Disk^n$. 
We start with $n$ independent radial $SLE$ paths. 
First, we will tilt independent $SLE$ by a loop term to define a process with the property that each of the $n$ paths locally acts like $SLE$ in the disk minus all $n$ initial segments. We will tilt this process again by another loop term and take a limit to give the definition of global multiple radial $SLE$. Splitting up the construction into two distinct tiltings will allow us to analyze the contribution of $t$-measurable loops separately from that of ``future loops."
Furthermore, each of these processes is interesting in its own right, and we show that in each case the driving function satisfies the radial Bessel equation. (See equations (\ref{Bessel-a}) and (\ref{eq:SDE2a}).)

This clarifies which terms cause the multiple paths to avoid each other's past versus the terms that ensure that the paths continue to avoid each other in the future until all curves reach the origin.

\subsection{Notation}

We will set up some basic notation; some of the notation that was used in the single $SLE$ setting above will be repurposed here in the setting of $n$ curves. (See Figure~\ref{fig:unzip}.)

\begin{figure}
	\begin{center}
		\includegraphics{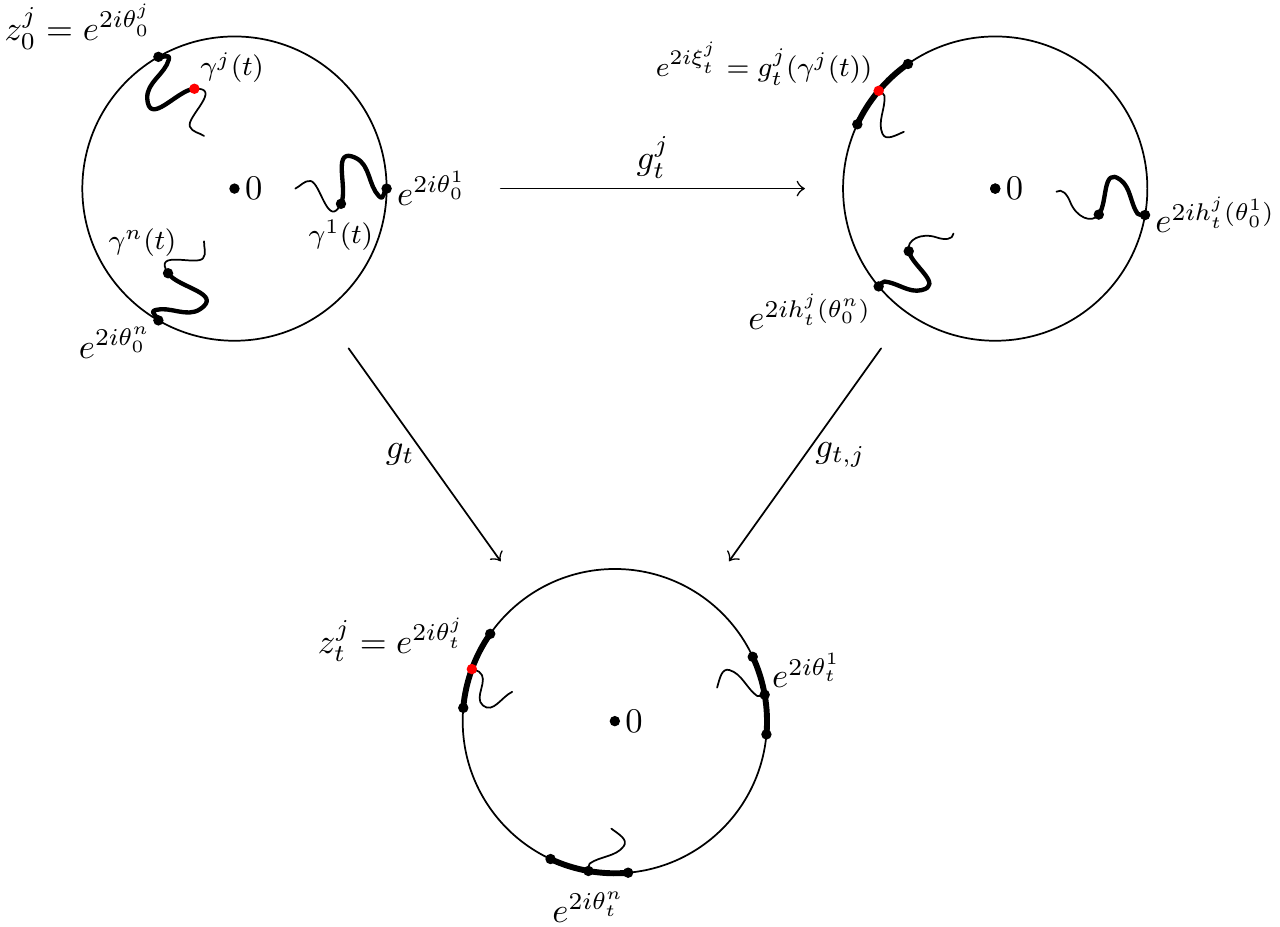}
	\end{center}
	\caption{$n$-radial $SLE$}
	\label{fig:unzip}
\end{figure}

\begin{itemize}

\item  We fix positive integer $n$ and let
$\btheta = (\theta^1,\ldots,\theta^n)$ with
\[     \theta^1
< \cdots < \theta^n  <  \theta^1 +\pi.\]
Let $z^j = \exp\{2i \theta^j$\} and $\bz = (z^1,\ldots,z^n)$.
Note that $z^1,\ldots,z^n$ are $n$ distinct points on
the unit circle ordered counterclockwise.

\item Let $\bgamma = (\gamma^1,\ldots,\gamma^n)$ be an
$n$-tuple of curves 
$\gamma^j:(0,\infty) \rightarrow \Disk\setminus\{0\}$
with $\gamma^j(0+) = z^j $ and 
$\gamma^j(\infty) = 0$.  We write $\gamma^j_t$ for
$\gamma^j[0,t]$ and $\bgamma_t = (\gamma^1_t,\ldots,\gamma^n_t)$.
 In a slight abuse of notation, we will
use $\gamma^j_t$ to refer to both the set $\gamma^j[0,t]$
and the   function
$\gamma^j$ restricted to times in $[0,t]$. 

\item  Let $  D_t^j, D_t$ be the connected components
of $  \Disk \setminus \gamma_t^j, 
   \Disk \setminus \bgamma_t,$  respectively, containing
   the origin.  Let $g_t^j:
   D_t^j \rightarrow \Disk, g_t:D_t \rightarrow \Disk$
be the unique conformal transformations with
\[   g_t^j(0) = g_t(0) = 0 , \;\;\;\;\;
      (g_t^j)'(0), g_t'(0) > 0.\]

\item Let $T$ be the first time $t$ such that $\gamma_t^j
\cap \gamma_t^k \neq \eset $ for some $1 \leq j < k \leq n$.

\item  Define $z_t^j = \exp\{2i\theta^j_t\}$ by $g_t(\gamma^j(t)) = z_t^j$.  Let $\z_t = (z^1_t,\ldots,z^n_t),
\btheta_t = (\theta^1_t,\ldots,\theta^n_t)$.
For $\zeta  \in \Half$ define $h_t(\zeta )$  to be
the continuous function of $t$ with $h_0(\zeta) = \zeta$ and 
\[     g_t(e^{2i\zeta}) = e^{2ih_t(\zeta)}.\]
Note that if $\zeta \in \R$ so that $e^{2i\zeta}  \in \p
\Disk,$ we can differentiate with respect to $\zeta$ to get
\begin{equation}  \label{jan25.3}
           |g_t'(e^{2i\zeta})|
           = h_t'(\zeta).  \end{equation}

\item  More generally, 
if $\bt = (t_1,\ldots,t_n)$ is an
$n$-tuple of times, we define $\bgamma_\bt,
D_\bt, g_\bt$.  We let
$        \alpha(\bt) = \log g_\bt'(0).$

\item We will say that the curves have the \textit{common
(capacity $a$-)parameterization}
if for each $t$,
\begin{equation}  \label{jan25.1}
   \p_{j}\alpha(t,t,\ldots,t) = 2a, \;\;\;\; j=1,\ldots,n.
   \end{equation}
In particular,
\begin{equation}  \label{jan25.2}       g_t'(0) = e^{2ant}. 
 \end{equation}
 Note that \eqref{jan25.1} is a stronger condition than
 \eqref{jan25.2}.

\end{itemize}

The following  form of the Loewner differential
equation is proved in the same was as the $n=1$ case,

\begin{proposition}\label{prop_radial_Loewner_equation}[Radial Loewner equation]
If $\bgamma_t$ has the common  parameterization, then
for $t < T$, the functions $g_t,h_t$ satisfy
\[        \dot g_t(w) = 2a \, g_t(w) \sum_{j=1}^n
              \frac{z_t^j + g_t(w)}{z_t^j - g_t(w)}, \;\;\;\;\;
  \dot h_t(\zeta) = a \sum_{j=1}^n \cot(h_t(\zeta) -
        \theta_t^j).\]
If $\p D_t$ contains an open arc of $\p \Disk$
including  $w = e^{2i\zeta}$, then
\begin{equation}  \label{jan25.4}
   |g_t'(w)| = \exp \left\{-a\int_0^t \sum_{j=1}^n \csc^2 (h_s(\zeta)-\theta^j_t)
 \, ds \right\}.
 \end{equation}
\end{proposition}

\subsection{Common parameterization and local independence  }\label{sec:locindep}

Suppose that $\gamma^1,\ldots,\gamma^n$ are independent
radial $SLE_\kappa$ paths in $\Disk$ starting at $z^1,\ldots,z^n$,
respectively, going to the origin.  Then we can
parameterize the paths so that they have the common
 parameterization. (This parameterization is only possible until the first time that two of the paths intersect, but this will not present a problem since we will usually restrict to nonintersecting paths.)  Indeed, suppose
 $\tilde \gamma^1,\ldots,\tilde \gamma^n$ are
independent $SLE_\kappa$ paths with the usual parameterization
as in Section \ref{radialsec}. It is not true that
$\tilde \bgamma_t = (\tilde \gamma^1,\ldots,\tilde \gamma^n_t)$ has the common  parameterizaton.  We will
write $\gamma^j(t) = \tilde \gamma(\sigma^j(t))$
where $\sigma^j(t)$ is the necessary time change. Define $g_{t,j}$ by $g_t =g_{t,j} \circ g_t^j$. The driving functions for $\tilde \bgamma_t$ are independent standard Brownian motions; denote these by $\drjradial{j}{t}$. Define $\drjcommon{j}{t}$ by 
$
\drjcommon{j}{t}=\drjradial{j}{\sigma^j(t)}
$
so that
$
e^{2 i \drjcommon{j}{t} }=g^j_t(\gamma^j(t)).
$ 
Furthermore, define
define $h_{t}^j$ and $h_{t, j}$ so that 
\begin{equation}\label{notation:htj}
h_t (w)=h_{t,j}\circ h^j_t(w),
\end{equation}
 and
\[ g_t^j (e^{2iw})= e^{2i h_{t}^j(w)}.\]
(See Figure \ref{fig:unzip}.)


\begin{lemma} \label{time_change_derivative}
	The derivative
	$\dot \sigma^j(t)$ depends only on $\bgamma_t$ and is given by
\begin{equation} \label{Time_change_common_param}
\dot \sigma^j(t) 
=h_{t,j}'(\xi^j_t)^{-2}.
\end{equation}
\end{lemma}

\begin{proof}
Differentiating both sides of equation (\ref{notation:htj}), we obtain
\begin{equation}\label{time_deriv_of_h}
\dot h_t(w)= \dot h_{t,j}(h^j_t(w)) +  h_{t,j}' (h^j_t(w)) \times \dot h^j_t(w).
\end{equation} 
Since $g^j_t$ satisfies the (single-slit) radial Loewner equation with an extra term for the time change,
$h^j_t$ satisfies 
\[ 
\dot h^j_t (w) = a \cot\left(h^j_t(w) - \drjcommon{j}{t} \right) \times \dot \sigma^j(t).
\]
On the other hand, $h_{t,j}$ satisfies
\[ \dot h_{t,j} (w) = a \sum_{k\neq j} \cot \left(h_{t,j}(w) - \theta^k_t\right).\]
Substituting these expressions for $\dot h_t(w)$ and $\dot h^j_t (w)$ into (\ref{time_deriv_of_h}) and using the equation for $\dot h_t(w)$ given in Proposition \ref{prop_radial_Loewner_equation} shows that
\begin{equation}
a \sum_{k=1}^n \cot \left( h_t(w)-\theta^k_t\right)
= a \sum_{k\neq j} \cot \left(h_{t}(w) - \theta^k_t\right)
+  h'_{t, j}(h^j_t(w)) \times a \dot \sigma^j(t)\cot\left(h^j_t(w) - \drjcommon{j}{t}\right).
\end{equation}
Solving for $\dot \sigma^j(t)$ and taking the limit as $w\to \gamma^j(t)$ verifies (\ref{Time_change_common_param}).
\end{proof}
 
 The components of $\bgamma$ are not quite independent because the rate of ``exploration'' of
 the path $\gamma^j$ depends on the other paths.
 However, the paths are still independent in the
 sense that the conditional distribution of the
 remainder of the paths given $\bgamma_t$ are
 independent $SLE$ paths; in the case of $\gamma^j$
 it is $SLE$ in $D_t^j $ from $\gamma^j(t)$ to $0$.
 
 We will define another process, which we will call \emph{locally independent} $SLE_\kappa$ that has the
 property that locally each curve grows like $SLE_\kappa$ from $\gamma^j(t)$ to $0$ in $D_t$  (rather than in $D_t^j$).  
 This will
 be done similarly as for a single path.  Intuitively, at time $t$ each curve $\gamma^j_t$ can ``see'' $\bgamma_t$, but not the future evolution of the other curves.
 
 Recall that
 $SLE_\kappa$ in $D\subset \Disk$ is obtained from $SLE_\kappa$ in $\Disk$ by weighting by the appropriate partition function. Since the partition function is not a martingale, this is done by finding an appropriate differentiable compensator so that the product is a martingale, and then applying Girsanov's theorem.
 
 Let
 \begin{equation}  \label{jan27.1}
   \slepart_t^j = \slepart(D_t^j, D_t;\gamma^j(t),0), \;\;\;\;
     \slepart_t = \prod_{j=1}^n \slepart_t^j,
     \end{equation}
 \begin{equation}  \label{oct3.1} \sumcsc(\btheta_t) =
        \sum_{j=1}^n \sum_{k \neq j} \csc^2(\theta^j_t - \theta^k_t),   \end{equation}
  \[
  \stoptime=\inf \{t: \exists j\neq k \text{ such that }\gamma^j_t \cap \gamma^k_t \neq \emptyset\}.
  \]

  For any loop, let 
\[  s^j(\ell) = \inf\{t: \gamma^j(t) \in \ell\},\;\;\;\;\;
   s(\ell) = \min_{j} s^j(\ell). \]
  We make a simple observation that will make the ensuing
  definitions valid.
\begin{lemma}
Let $\gamma^1,\ldots,\gamma^n$ be nonintersecting
curves.
Then except for a set of loops of Brownian loop measure zero, either $s(\ell) =\infty$
or there exists a unique $j$ with $s^j(\ell)
= s(\ell).$
\end{lemma}

\begin{proof}[Sketch of proof]  We consider excursions between
the curves $\gamma^1,\ldots, \gamma^n$, that is,
times $r$ such that $\ell(r) \in \gamma^k$ for
some $k$ and the most recent visit before time
$r$ was to a different curve $\gamma^j$.  There
are only a finite number of such excursions.   For each one,
the probability of hitting a point with the current
smallest index is zero.
\end{proof}

Let $\pastloops{j}{t} = \pastloops{j}{t}(\bgamma_t)$ be the  set
of loops $\ell$ with $s(\ell) < s^j(\ell)
  \leq t$,   and let
\begin{equation} \label{defpastloops}
   \pastloopterm_t = \pastindicator_t \, \exp \left\{ \frac{\cent}{2}
    \sum_{j=1}^n m_\Disk(\pastloops{j}{t}) \right\} .
\end{equation}
(See Figure \ref{fig:loops}.) Here $\hat{I}_t$ is the indicator function that
$\gamma^j_t \cap \gamma_t^k = \eset $
for $j \neq k $.


We note that while the definitions of $s^j$
and $s$ (and hence $\pastloops{j}{t}$) depend on the parameterization
of the curve, $\pastloopterm_t$ depends only on
the traces of the
curves $\gamma^1_t,\ldots,\gamma^n_t$.
For this reason, we could also define $\pastloopterm_\bt$
for an $n$-tuple $\bt = (t_1,\ldots,t_n)$.

\begin{proposition}\label{locally_independent}  
	Let $0 < \kappa \leq 4$.
	If $\bgamma_t$ is independent
$SLE_\kappa$ with the   common parameterization, and
\begin{equation}\label{def:locindepmart}
   M_t = \pastloopterm_t
\, \Psi_t \, \exp\left\{ab\int_0^t \sumcsc(\btheta_s) \,ds
\right\},
\end{equation}
  then $M_t$ is a local martingale for $0 \leq t < \stoptime$.
 If $\Prob_*$ denotes the measure obtained by
 tilting $\Prob$ by $M_t$, then
 \begin{equation}\label{Bessel-a}
     d \theta_t^j = a \sum_{k \neq j}
    \cot(\theta_t^j -\theta_t^k) \, dt + dW_t^j,
 \end{equation}
 where $W_t^1,\ldots,W_t^n$ are independent
 standard Brownian motions with respect to
 $\Prob_*$.  Furthermore,
 \[   \Prob_*\{\stoptime < \infty\} = 0 . \]
\end{proposition}

\begin{definition}We call the $n$-tuple of curves $\bgamma_t$ under the measure $\Prob_*$ \emph{locally independent} $SLE_\kappa$.
\end{definition}

The idea of the proof will be to express $M_t$ as a product of martingales
 \[M_t = \prod_{j=1}^n M^j_t\] with the following property: after tilting by the martingale $M_t^j$ the curve $\gamma^j$ locally at time $t$ evolves as $SLE_\kappa$ in the domain $D_t=\Disk\setminus \bgamma_t$. The martingales $M^j_t$ are found by following the method of proof in [Proposition 5, \cite{Mohammad}]. The construction shows that under $\mathbb P_*$, at each time $t$ the curves $\gamma^1, \ldots \gamma^n$ are locally growing as $n$ independent $SLE_\kappa$ curves in $D_t$,  
  which is the reason for the name \emph{locally independent $SLE$}. Locally independent $SLE$ is revisited in \S\ref{sec:locallyindepSLE}.

\begin{proof}[Proof of Proposition \ref{locally_independent}]
Since the $\drjcommon{j}{t}$ are independent standard Brownian motions under the time changes $\sigma^1, \ldots, \sigma^n$, there exist independent standard Brownian motions $B^1_t, \ldots, B^n_t$ such that 
\[d\drjcommon{j}{t}= \sqrt {\dot \sigma^j(t)} \,dB^j_t, \qquad j=1, \ldots, n.
\]
By Lemma \ref{time_change_derivative},
\[
dB^j_t=h_{t,j}' \,(\drjcommon{j}{t})\,d\drjcommon{j}{t} , \quad j=1, \dots, n, 
\]
and It\^o's formula shows that each $\theta^j_t$ satisfies
\[
d \theta^j_t = \dot h_{t,j}(\drjcommon{j}{t}) \, dt + \frac{h_{t,j}''(\drjcommon{j}{t})}{2 \left( h_{t,j}'(\drjcommon{j}{t}) \right)^2} \, dt + dB^j_t. 
\]
Define
\begin{equation}
M^j_t=I^j_t \slepart_t^j  \exp \left\{ \frac{\cent}{2} m_\Disk(\pastloops{j}{t}) \right\}
\exp\left\{  ab \int _0^t \sum_{k\neq j} \csc^2 (\theta^j_s -\theta^k_s) \,ds \right\}, \quad t<T,
\end{equation}
so that
\begin{equation}
M_t=\prod_{j=1}^n M^j_t.
\end{equation}
Applying the method of proof of the boundary perturbation property for single slit radial $SLE$ [Proposition 5, \cite{Mohammad}], we see that $\slepart_t^j$ satisfies
\[
d\slepart^j_t=\slepart^j_t \left[ \left( -\frac{\cent}{2}m_\Disk (\pastloops{j}{t}) - ab \int_0^t \sum_{k:k\neq j} \csc^2 (\theta^j_s-\theta^k_s)ds \right) dt +  \frac{b}{2} \frac{h''_{t,j} (\drjcommon{j}{t})}{h'_{t,j} (\drjcommon{j}{t})} \, d\drjcommon{j}{t}\right],
\]
and $M^j_t$ is a local martingale satisfying
\[
dM^j_t= M^j_t  \frac{b}{2} \frac{h''_{t,j} (\drjcommon{j}{t})}{h'_{t,j} (\drjcommon{j}{t})} \, d\drjcommon{j}{t}, \quad M^j_0=1.
\]
Since the $\drjcommon{j}{t}$ are independent,
$M_t$ satisfies
\[
dM_t= M_t\left[ \sum_{j=1}^n  \frac{b}{2} \frac{h''_{t,j} (\drjcommon{j}{t})}{h'_{t,j} (\drjcommon{j}{t})} \, d\drjcommon{j}{t}  \right].
\]
Therefore, $M_t$ is a local martingale, and equation (\ref{Bessel-a}) follows by the Girsanov theorem. 
\end{proof}

\subsection{Dyson Brownian Motion on the Circle}

The construction of  $n$-radial $SLE_\kappa$ in Section \ref{sec:n-radial} will require some results about the $n$-radial Bessel process (Dyson Brownian motion on the circle), which we state here.  However, the proofs of these results are postponed until Section \ref{sec:DysonBM}, since they hold in the more general setting of $0<\kappa< 8$ and do not rely on Brownian loop measure. 

A note about parameters: we state the results here using parameters $\alpha$ and $b_\alpha$ since the results hold outside of the $SLE$ setting. When we apply these results to $SLE_\kappa$ in the next section, we will set $\alpha=a=2/\kappa$ or $\alpha=2a=4/\kappa$. In particular, when $\alpha=a$,  $b_\alpha=b=(3a-1)/2$.

Define
  \[ F_\alpha (\btheta) =  \prod_{1\leq j<k\leq n} |\sin(\theta^k - \theta^j)|^\alpha
,\;\;\;\; 
\stoptime=\inf \{ t: F_\alpha(\btheta) =0\},
\]
and recall the definition of $\sumcsc(\btheta)$ from \eqref{oct3.1}. The next result will be verified in the discussion following the proof of Lemma \ref{lemma:trig_identity}.

\begin{proposition} \label{Mart_a}
Let $\theta^1, \ldots, \theta^n$ be independent standard Brownian motions, and let $\alpha>0$. If
\begin{equation}\label{eq:defMart_a}
\begin{aligned}
M_{t,\alpha}
 &= F_\alpha( \btheta_t) \exp \left \{ \frac{\alpha^2 n(n^2-1)}{6}t \right \} \, \exp \left\{ \frac{\alpha-\alpha^2}{2} \int_0^t \psi( \btheta_s) ds \right \}, \quad 0\leq t<\tau,
\end{aligned}
\end{equation}
then $M_{t,\alpha}$  is a local martingale for $0 \leq t < \tau$
satisfying
\[
dM_{t,\alpha}=M_{t,\alpha}\,\sum_{j=1}^n \left( \sum_{k\neq j} \alpha \cot(\theta^j_t-\theta^k_t) \right) d\theta^j_t.
\]
If $\Prob_\alpha$ denotes the probability measure obtained after tilting by $M_{t,\alpha}$, then
\begin{equation}   \label{aBessel}
d\theta^j_t=\alpha\sum_{k\neq j}\cot(\theta^j_t-\theta^k_t)\,dt + dW^j_t,\;\;\;\; 0\leq t < \tau, 
\end{equation}
where $W^1_t, \ldots, W^n_t$ are independent standard Brownian motions with respect to $\Prob_\alpha$. Furthermore, if $\alpha \geq 1/2$, 
\[\mathbb P_\alpha (\tau=\infty)=1.\]
\end{proposition}

%

 \begin{proposition} \label{prop:P2a}
 	Suppose that $\alpha \geq 1/4$ and
	\begin{equation} \label{def:N}
	N_t = N_{t,\alpha,2\alpha} =  F_\alpha(\btheta_t) \,
	\exp \left \{ \frac{  \alpha^2 n(n^2-1)}{2}t \right \} \exp \left\{ -\alpha b_\alpha\int_0^t \psi( \btheta_s) ds \right \} 
	,
	\end{equation}
	where $b_\alpha = (3\alpha-1)/2$. Then $N_t$ is a $\Prob_\alpha$-martingale, and the measure
	obtained by tilting $\Prob_\alpha$ by $N_t$ is $\Prob_{2\alpha}$.
\end{proposition}

\begin{proof}
	See Proposition \ref{prop:P2a_restated} and its proof.
\end{proof}

We will also require the following theorem, which is proven immediately after Proposition \ref{prop:P2a_restated}.

\begin{theorem} \label{exponentialrateofconv1}
	If $\alpha \geq 1/2$, there exists $u=u(2\alpha,n) > 0$ such that 
	\[ \E^{\btheta_0}_\alpha\left[ \exp \left\{ -\alpha b_\alpha \int_0^t \psi( \btheta_s) ds \right \} \right] =  e^{-2an\beta t} \, F_\alpha(\btheta_t)  \,\frac{\Integral{3a}}{
		\Integral{4a}}\,
	[ 1 + O(e^{-ut})],  \]
	where
	\begin{equation}\label{eq:defbeta}
	 \beta = \beta(\alpha,n)  = \frac{\alpha(n^2-1)}{4},
	 \end{equation}
	 and $\E_\alpha $ denotes expectation with respect to $\Prob_\alpha$.
\end{theorem}

\subsection{$n$-Radial $SLE_\kappa$} \label{sec:n-radial}
The remainder of the section is devoted to the construction of \emph{$n$-radial $SLE_\kappa$}, which may also be called \emph{global multiple radial} $SLE_\kappa$. As we have
stated before, we will consider three measures on $n$-tuples of curves with the common parameterization.
\begin{itemize}
\item $\Prob,\E$ will denote independent $SLE_\kappa$ 
with the common parameterization;
\item $\Prob_* , \E_*$ will denote locally independent 
$SLE_\kappa$;
\item $\multProb, \multE$ will denote $n$-radial
$SLE_\kappa$.
\end{itemize}

In Section \ref{sec:locindep}, we obtained $\Prob_*$ from $\Prob$ by tilting by a
$\Prob$-local martingale $M_t$.
We will obtain $\multProb$
from $\Prob_{*}$ by tilting by a $\Prob_*$-local
martingale $N_{t,T}$ and then letting
$T \rightarrow \infty$.  Equivalently, we obtain
$\multProb$ from $\Prob$ by tilting by 
$\tildeN{t}{T} := M_t N_{t,T}$ and letting
$T \rightarrow \infty$.

%
%
%

Let $\loops{j}{t} = \loops{j}{t}(\bgamma_t)$ be the set
of loops $\ell$ with $s(\ell) < s^j(\ell)$ and $s(\ell) \leq t$, as in Figure \ref{fig:loops}.  Define
\[   \loopterm_t = \indicator_t \, \exp \left\{ \frac{\cent}{2}
\sum_{j=1}^n m_\Disk(\loops{j}{t}) \right\} .\] 
Here $I_t$ is the indicator function that
$\gamma^j_t \cap \gamma^k = \eset $
for $j \neq k $. 

Let
\[   
\tildeN{t}{T} =  \E^{\btheta_0}\left[
\loopterm_T \mid \bgamma_t \right],\;\;\;\;
  0 \leq t \leq  T,\]
 where the conditional expectation is with respect
 to $\Prob$.  By construction, $\tildeN{t}{T}$ is a martingale for $0\leq t\leq T$ with $\tildeN{0}{T} = \E^{\btheta_0}[\loopterm_T]$.

For the next proposition, recall that $\pastloopterm_t$ weights by loops that hit at least two curves before time $t$; the precise definition is given in (\ref{defpastloops}). 
 
 \begin{proposition} 
 	Let $T\geq 0$. If $\bgamma_t$ is independent
$SLE_\kappa, 0 < \kappa \leq 4$,  with the  common parameterization,  then
 \begin{equation}  \label{condition}
\tildeN{t}{T} =  \pastloopterm_t \, \slepart_t  \, \E^{\btheta_t}\left[\loopterm_{T-t}\right], \quad 0\leq t \leq T. \end{equation}
 In particular, if
 \[     N_{t,T} =  \exp\left\{-ab\int_0^t \sumcsc(\btheta_s) \,ds
  \right\} \, \E^{\btheta_t}\left[\loopterm_{T-t}\right], \quad 0\leq t \leq T, \]
 then $N_{t,T}$ is a $\Prob_*$-martingale for $0 \leq t \leq  T$,
 and 
   \begin{equation}\label{ELoop}
 \E^{\btheta_t}\left[\loopterm_{T-t}\right]
  =  \E^{\btheta_t}_*\left[  \exp\left\{-ab\int_0^{T-t} \sumcsc(\btheta_s) \,ds
  \right\} \right].
  \end{equation}
  \end{proposition}
Note that the expectation on the righthand side of \eqref{ELoop} is with respect
to $\Prob_*$. 

\begin{proof}
We may write
\[
\loopterm_T= \pastloopterm_t\, \frac{\loopterm_t}{\pastloopterm_t} \, \loopterm_{T,t},
\]
where
\[
\loopterm_{T,t} =\exp\left\{\frac{\cent}{2}\sum_{j=1}^nm_\Disk (\ell: t<s(\ell)\leq T, s(\ell)<s^j(\ell))  \right\},
\]
The term $\loopterm_{T,t}$ should be thought of as the ``future loop'' term, since it accounts for loops that hit at least two curves with the first hit occurring during $(t,T]$. 

The restriction property shows that
\[
\E^{\btheta_0} \left[ \pastloopterm_t  \left (\frac{\loopterm_t}{\pastloopterm_t}\right) \, \big \vert \, \bgamma_t \right]
= \pastloopterm_t \,\slepart_t.
\]
Moreover, the conditional distribution on $\bgamma_T\setminus \bgamma_t$, after tilting by $\pastloopterm_t  \left ({\loopterm_t}/{\pastloopterm_t}\right) $ is that of
independent $SLE$ in $D_t$. 
Since $\loopterm_{T,t}$ depends only on $\bgamma_T\setminus \bgamma_t$,
this gives  (\ref {condition}).

For the second part of the proposition, notice that 
\[ N_{t,T}\, M_t  = \tildeN{t}{T},\]
 which is a $\Prob$-martingale by construction, so $N_{t,T}$ is a $\Prob_{*}$-martingale. Since 
 $\E^{\btheta_T}\left[ \loopterm_0\right]=1$, this implies that 
\[
N_{t,T}=\E_*^{\btheta_0}\left[ N_{T,T} \mid \bgamma_t \right]
= \E_*^{\btheta_0}\left[  \exp \left \{ -ab \int_0^T \sumcsc (\btheta_s)\, ds \right \}  \,\vert\, \bgamma_t \right],
\]
which verifies (\ref{ELoop}). 
\end{proof}

\begin{proposition}\label{conditional_distr_T}
Let $\Prob_*^T$ denote the probability measure obtained by tilting $\Prob$ by $\tildeN{t}{T}$. Under $\Prob_*^T$, conditionally on $\hat \bgamma^j_T$, the distribution of $\gamma^j$ is $SLE_\kappa$ in $\mathbb D\setminus \hat \bgamma^j_T$.
\end{proposition}

\begin{proof}
	The result follows by an application of the restriction property.
\end{proof}

The next result, which gives the exponential rate of convergence of $\E^{\btheta_0}[\loopterm_T]$, is a direct application of Theorem \ref{exponentialrateofconv1}.

\begin{corollary} \label{expconvloopversion} 
	There exists $u=u(2a,n)>0$ such
	that as $T \rightarrow \infty$,
	\[ 
	\E^{\btheta_0}\left[ \loopterm_T\right] = 
	\frac{\Integral{3a}}{\Integral{4a}}\,
	F_a(\btheta_0) \, e^{-2an \beta T}
	\, \left [1 + O(e^{-uT})\right],\]
		where $\beta$ is given by (\ref{eq:defbeta}).
\end{corollary}

\begin{proof}
	Notice that (\ref{ELoop}) implies that
	\begin{equation}\label{ELoop2}
	\tildeN{0}{T}=
	\E^{\btheta_0}\left[
	\loopterm_T \right] = \E^{\btheta_0}_*\left[  \exp\left\{-ab\int_0^{T} \sumcsc(\btheta_s) \,ds
	\right\} \right]. 
	\end{equation}
	Substituting this into Theorem \ref{exponentialrateofconv1} gives the result.
\end{proof}
We define the \emph{$n$-interior scaling exponent}:
\begin{equation}\label{eq:n_scalingexp}
\hat \beta_n = \beta -\tilde b(n-1) = \frac{4(n^2-1) + (6-\kappa)(\kappa-2)}{8\kappa},
\end{equation}
where $\beta$ is defined by (\ref{eq:defbeta}).

\begin{proposition} \label{prop:tildeM}  
	With respect to $\Prob$,
	\[    \tilde M_t :=  e^{2an\hat \beta_n t}\, \pastloopterm_t\,
	F_a(\btheta_t)\]
	is a local martingale.
	If $\multProb$ denotes the measure obtained by
	tilting by $\tilde M_t$, then
	\begin{equation}\label{eq:SDE2a}
	d \theta_t^j = 2a \sum_{k \neq j}
	\cot(\theta_t^j -\theta_t^k) \, dt + dW_t^j,
	\end{equation}
	where $W_t^1,\ldots,W_t^n$ are independent
	standard Brownian motions with respect to
	$\multProb$. 
\end{proposition}

\begin{proof}
	Comparing (\ref{Bessel-a}) and (\ref{aBessel}), we see that tilting $n$ independent Brownian motions by $M_{t, a}$ gives the SDE satisfied by the driving functions of locally independent $SLE_\kappa$. 
	By Proposition \ref{prop:P2a}, tilting further by $N_{t,a,2a}$ (defined in (\ref{def:N}) above) gives driving functions that satisfy (\ref{eq:SDE2a}), which is the $n$-radial Bessel equation (\ref{aBessel}) for $\alpha=2a$. 
	This implies that $\Prob_{2a}$ is obtained by tilting $\Prob$ by $M_t N_{t,a,2a}$. 
	
	To verify that 	\[
	M_t N_{t,a,2a} = \tilde M_t,
	\]
	we use the fact that
	\[
	\slepart_t=\exp\left\{ -2a\tilde b n(n-1)t \right\}.
	\]
	which follows from conformal covariance of the partition function.
\end{proof}

As above, let $\Prob$ denote the measure on $n$ independent radial $SLE_\kappa$ curves from $\btheta_0$ to $ 0$ with the $a$-common parameterization.
\begin{theorem}\label{maintheorem} 
	Let $0 < \kappa \leq 4$.  
Let $t>0$ be fixed. For each $T>t$, let 
	$\mu_{T}=\mu_{T,t}$ denote the measure
	whose Radon-Nikodym
	derivative with respect to $\Prob$ is
	\[                  \frac{  \loopterm_T}
	{ \E^{\btheta_0}\left[  \loopterm_T\right]}.\]
	Then as $T \rightarrow \infty$, the measure $\mu_{T,t}$
	approaches $\multProb$ with respect to the variation distance. Furthermore, the driving functions $z^j_t=e^{2i \theta^j_t}$ satisfy
	\begin{equation}
	d \theta_t^j = 2a \sum_{k \neq j}
	\cot(\theta_t^j -\theta_t^k) \, dt + dW_t^j,
	\end{equation}
	where $W^j_t$ are independent standard Brownian motions in $\multProb$.
\end{theorem}

%
\begin{proof}
	We see that
	\begin{equation}
	\begin{aligned}
	\frac{d \mu_{T,t}}{d\Prob_t} =  \frac{\E^{\btheta_0} \left[ \loopterm_T  \mid \bgamma_t \right]}{ \E^{\btheta_0}\left[  \loopterm_T\right]}
	&=\frac{\tildeN{t}{T}}{\tildeN{0}{T}}.
	\end{aligned}\end{equation}
	By Proposition \ref{prop:tildeM}, $\multProb$ is obtained by tilting $\Prob$ by $\tilde M_t$, so we compare $\tildeN{t}{T}$ to $\tilde M_t$ and apply Corollary \ref{expconvloopversion}:
	
	\begin{equation}
	\begin{aligned}
	\frac{d \mu_{T,t}/d \Prob_t}{ d\multProb_{t}/d\Prob_t}
	&= \frac{\tildeN{t}{T}/\tildeN{0}{T}}{ \tilde M_t / \tilde M_0} \\
	&= \frac{\E^{\btheta_t}\left[ \loopterm_{T-t}\right] F_a(\btheta_0) }{\E^{\btheta_0}\left[ \loopterm_T\right] F_a(\btheta_t) \exp \left\{ \frac{a^2 n(n^2-1)}{2}t\right\} }
 \frac{ }	{  }\\
	&= 1+O (e^{-u(T-t)}).
	\end{aligned}
	\end{equation}
%
%
Therefore,
	\[
	\lim_{T\to \infty} \left[
	\frac{d \multProb_{t}}{d \Prob_t} \left(\frac{d \mu_{T,t}} {d\Prob_t} -\frac{d\multProb_{t}}{d \Prob_t} \right) \right]=0.
	\]
But $\frac{d \multProb_{t}}{d \Prob_t}$ is constant (since $t$ is fixed), so this implies convergence of $\mu_{T,t}$ to $\multProb_{t}$ in the variation distance.

\end{proof}

\begin{definition} Let $0<\kappa\leq 4$.
	If the curves $\gamma^1, \ldots, \gamma^n$ are distributed according to $\multProb$, we call $\bgamma$ \emph{(global) $n$-radial $SLE_\kappa$}.
\end{definition}

\begin{corollary}
	Let $\bgamma$ be $n$-radial $SLE_\kappa$ for $0<\kappa \leq 4$. With probability one, $\bgamma$ is an $n$-tuple of simple curves.
\end{corollary}

\begin{proof}
By construction, $n$-radial $SLE_\kappa$ is a measure on $n$-tuples of curves that is absolutely continuous with respect to $n$-independent $SLE_\kappa$. But since $0<\kappa\leq 4$, each independent $SLE_\kappa$ curve is almost surely simple.
\end{proof}

To conclude this section, we remark that the results above do not address the question of continuity at $t=\infty$. Additionally, it would be natural to extend the definition of $n$-radial $SLE$ to apply to $\kappa \in (0, 8)$ by using the measure $\Prob_{2a}$ instead of $\multProb$, but we will not consider this here.

\section{Locally independent $SLE$} \label{sec:locallyindepSLE}

Here we discuss locally independent $SLE$  and explain how it arises as a limit of processes that act like ``independent $SLE$ paths in the current domain.''
For ease we will do the chordal case and $2$ paths, but the same idea works for any number of paths and for radial $SLE.$ Locally independent $SLE$ is defined here for all $\kappa<8$, but when $\kappa \leq 4$ the radial version is the same as the process defined in Proposition \ref{locally_independent}. 

This construction clarifies the connection between locally independent $SLE$ and commuting $SLE$ defined in \cite{Dubedat}. Intuitively, given a sequence of commuting $SLE$ increments, as the time duration of the increments goes to $0$, the curves converge to locally independent $SLE$.

Throughout this section we write $\B_t = (B_t^1,B_t^2)$
for a standard two-dimensional Brownian motion, that is, two independent one-dimensional Brownian motions.
We will use the fact that $\B_t$ is H\"older continuous. We give a quantitative version here
which is stronger than we need. 

\begin{itemize}
	\item  Let $E_h$ denote the event that for all
	$0 \leq t \leq 1/h$ and all $0 \leq s \leq h$,
	\[       |\B_{t+s} - \B_t|
	\leq s^{1/2} \, \log(1/s) . \]
	Then as $h \rightarrow 0$, $\Prob(E_h^c)$ decays
	faster than any power of $h$.
\end{itemize}

We will define the discrete approximation using the same Brownian motions as for the continuum and then the convergence follows from deterministic estimates coming from the Loewner equation.  Since these are standard we will not give full details.
We first define the process.
Let $a = 2/\kappa$.

\begin{definition}  Let $\X_t =(X_t^1,X_t^2)$ be the
	solution to the SDEs,
	\[  dX_t^1 = \frac{a}{X_t^1 - X_t^2} \, dt
	+ dB_t^1, \;\;\;\;\;
	dX^2_t = \frac{a}{X_t^2 - X_t^1} \, dt
	+ dB_t^2 , \]
	with $X_0^1 = x_1, X_0^2 = x_2$. 
	Let $\tau_u = \inf\{t: |X_t^2- X_t^1| \leq u\}$,
	$\tau = \tau_{0+}$.  
\end{definition}

Note that $Z_t :=  X_t^2 - X_t^1 $ satisfies
\[       dZ_t = \frac{ 2a}{Z_t} \,dt + \sqrt 2 \,
dW_t, \]
where $W_t := (B_t^2 - B_t^1)/\sqrt 2$ is a
standard Brownian motion. This is a (time change of a)
Bessel process from which we see that $\Prob\{\tau < \infty\} = 0$ if and only if $\kappa \leq 4$. 
If $4 < \kappa < 8$ we can continue the process for all $\tau < \infty$ by using reflection.  We will consider
only $ \kappa < 8$.

\begin{definition}  If $\kappa < 8$,
	locally independent $SLE_\kappa$
	is defined to be the collection of conformal maps $g_t$ satisfying the Loewner equation
	\[    \p_t g_t(z) = \frac{a}{g_t(z) - X_t^1}
	+ \frac{a}{g_t(z) - X_t^2} , \;\;\;\;
	g_0(z) = z. \]
	This is defined up to time
	\[        T_z = \sup\{t: \Im[g_t(z)] > 0\}. \]
\end{definition}

Locally independent $SLE_\kappa$ produces a
pair of curves $\bgamma(t) = (\gamma^1(t),\gamma^2(t)).$   Note that $\hcap[\bgamma_t] = 2at$.  If $
\kappa \leq 4$, then $\gamma^1_t \cap \gamma^2_t \neq \eset$; this is not true for all $t < \tau$ if
$4 < \kappa < 8$.

Let us fix a small number $h = 1/n$ and consider the
process viewed at time increments $\{kh:k=0,1,\ldots\}$. The following estimates  hold uniformly for $k \leq 1/h$
on the event $E_h$.   The first comes just by the definition of the SDE and the second uses the
Loewner equation.  Let
$\Delta_k^j = \Delta_{k,h}^j  = B^j_{ kh} - B_{(k-1)h}^j$.
\begin{itemize}
	\item  If $|X_{hk}^2 - X_{hk}^1 | \geq h^{1/8}$,
	then
	\begin{equation}  \label{1.1}
	X_{ (k+1)h}^j = X_{kh}^j
	+ \frac{ah}{X_{kh}^j - X_{kh}^{3-j}}
	+ \Delta_{k+1}^j
	+ o(h^{4/3}). 
	\end{equation}
	\item If $\Im[g_{hk}(z)] \geq u/2$, 
	and $ 0 \leq s \leq h$,
	\begin{equation}  \label{1.2}  g_{kh+s}(z) = g_{kh}(z)
	+ \frac{as}{g_{kh}(z)- X_{kh}^1}
	+ \frac{as}{g_{hk}(z)- X_{kh}^2}
	+ o_u(h^{4/3}). 
	\end{equation}
\end{itemize}

%
%
%

We will compare this process to the process which
at each time $kh$ grows independent $SLE_\kappa$
paths in the current domain, increasing the capacity of each path  by $h$.  Let us start with
the first time period in which we have independent
$SLE$ paths.  Again, we restrict to the event $E_h$.
\begin{itemize}
	\item Let $\tilde \gamma^1,\tilde \gamma^2$ be independent $SLE_\kappa$
	paths starting at $x_1,x_2$ respectively with driving
	function $\tilde X^j_t =   B^j_t$,  each run until
	time $h$.  To be more precise if $\tilde g^j_t:
	\Half \setminus \gamma^j_t \rightarrow \Half$
	is the standard conformal transformation, then
	\[       \p_t \tilde g^j_t(z)
	= \frac{a}{g^j_t(z) - \tilde X^j_t}, \;\;\;\;
	g^j_0(z) = z, \;\;\;\; 0 \leq t \leq h\]
	Note that   $\hcap[\gamma^1_t]
	= \hcap[\gamma^2_t] = ah$. Although $\hcap[\bgamma_t]
	< 2ah$,  if $|x_2 - x_1| \geq h^{1/8}$, 
	\[  \hcap[\bgamma_t] = 2ah - o(h^{4/3}).\]
	This defines $\bgamma_t$ for $0 \leq t \leq h$ and
	we get corresponding conformal maps
	\[   \tilde g_t    :
	\Half \setminus \bgamma_t \rightarrow
	\Half, \;\;\;\; 0 \leq t \leq h .\]
	If $\Im[ z] \ \geq 1/2$, then
	\[  \tilde g_{h}(z) = z + \frac{ah}
	{z - x^1} + \frac{ah}{z-x^2}
	+ o(h^{4/3}).\]
	Also, by writing $\tilde g_h =  \phi \circ  \tilde  g_h^j$,
	we can show  that
	\[   \hat X_h^j := \tilde g_h(\tilde \gamma^j(h)) = 
	\tilde X_h^j + \frac{ah}{x_j - x_{3-j}}
	+ o(h^{4/3}).\]
	
	\item Recursively, given $\hat X_{kh}$
	and  $\tilde \gamma_t^1,
	\tilde \gamma_t^2$ and $\tilde g_t$ for
	$0 \leq t \leq kh$ (the definition of these quantities depends on
	$h$ but we have suppressed that from the notation),  let 
	\[ \tilde X_{kh+t}^j
	= \hat X_{kh}^j + [B_{kh + t}^j -
	B_{k}^j], \;\;\;\; 0 \leq t \leq h , \]
	and let 
	$\hat \gamma_{t,k}^1,
	\hat \gamma_{t,k}^2, 0 \leq t \leq h$ be independent $SLE_\kappa$
	paths with driving functions $\tilde X_{kh+t}^j$.
	For $j=1,2$, define 
	\[      \tilde \gamma_{kh+t}^j
	=   \tilde g_{kh}^{-1} [\hat \gamma_{t,k}^j], 
	\;\;\;\; 0 \leq t \leq h.\]
	This defines $\tilde \bgamma_{kh+t}, 0 \leq t \leq h$
	and $\tilde g_{kh+t}: \Half \setminus \bgamma_{kh+t}
	\rightarrow \Half$ is defined as before. Set
	\[   \hat X_{(k+1)h}^j = \tilde g_{(k+1)h}(\tilde{\gamma}^j((k+1)h)).\]
	Note that if $|\hat X_{kh}^2 - \hat X_{kh}^1| 
	\geq h^{1/8}$, then 
	\begin{equation}  \label{1.3}    \hat X_{(k+1)h}^j
	= \hat X_{kh}^j + \Delta^j_{(k+1)h}
	+ \frac{ah}{ \hat X_{kh}^j - \hat X_{kh}^{3-j}} + o(h^{4/3}). \end{equation}
	Also, if $\Im[\tilde g_{kh}(z)] \geq u/2$,
	\begin{equation}  \label{1.4}
	\tilde g_{(k+1)h}(z) =
	\tilde g_{kh}(z) +  \frac{ah}{ \tilde g_{kh}(z)
		- \hat X_{kh}^1 } + \frac{ah}{ \tilde g_{kh}(z)
		- \hat X_{kh}^2 }
	+ o_u(h^{4/3}).
	\end{equation}
	
	\item If at any time $\tilde \gamma_{h,k}^1
	\cap \tilde \gamma_{h,k}^2 \neq \eset$ this procedure
	is stopped.
\end{itemize}

Note that  we are using the same Brownian motions as
we used before.  

\begin{proposition} 
	With probability one, for all $t < \tau$ and all
	$z \in \Half \setminus \bgamma_t$,
	\[   \lim_{h \downarrow 0} \tilde g_t(z)
	= g_t(z) . \]
\end{proposition}

\begin{proof}  We actually prove more.  Let
	\[   K(u,h) =
	\sup\left\{|\tilde g_t(z) - g_t(z)|:
	\Im[g_t(z)] \geq u, t \leq \tau_u \wedge u^{-1}
	\right\}. \]
	Then for each $u > 0$, with probability one,
	\[  \lim_{h \downarrow 0} K(u,h) = 0 . \]
	
	We fix $u$ and allow constants to depend on $u$
	and assume that $\Im[g_t(z)] \geq u$. 
	Then, if 
	\[     \Theta_k = \max_{r \leq k  }
	|\X_{rh}^j - \tilde \X_{rh}^j| \]
	Then \eqref{1.1} and \eqref{1.3} imply that
	\[  \Theta_{k+1} \leq \Theta_k [1 + O(h)] +
	O(h^{4/3} ) , \]
	or if $\hat \Theta_k = \Theta_k + k \, h^{4/3}, $
	then 
	$ \hat \Theta_k \leq \hat \Theta_k[1 + O(h)]
	.$   This shows that $\hat \Theta_k$ is bounded
	for $k \leq (hu)^{-1}$ and hence
	\begin{equation}  \label{1.5}
	\Theta_k \leq c h^{1/3}, \;\;\;\;
	k \leq (hu)^{-1}.
	\end{equation}
	We now let
	\[         D_k = \max_{r \leq k}
	|g_{kh}(z) - \tilde g_{kh}(z) | , \]
	and see that \eqref{1.2}, \eqref{1.4}, and \eqref{1.5}
	imply 
	\[     D_{k+1} \leq D_k [1 + O(h)] +
	O(h^{4/3} ) , \]
	which then gives
	\[            D_k \leq c h^{1/3}, \;\;\;\;
	k \leq (hu)^{-1}.\]
	Note that for $kh \leq t \leq (k+1)h$.
	\[  g_t(z) = g_{kh}(z)  +O(h), \;\;\;\;
	\tilde g_t(z) = \tilde g_{kh}(z) + O(h), \]
	and hence for all $t \leq \tau_u \wedge u^{-1}$
	\[   \tilde g_t(z) = g_t(z) + O(h^{1/3}).\]

\end{proof}

\section{$n$-Radial Bessel process}\label{sec:DysonBM}

In this section we study
  the process that we call the $n$-particle radial Bessel process. The image of this process under the
  map $z \mapsto e^{2iz}$ will be called Dyson
  Brownian motion on the circle.   We fix integer $n\geq 2$ and allow constants to depend on $n$.   Let ${\mathcal X}'
 = {\mathcal X}_n' $ be the torus $[0, \pi)^n$ with periodic boundary conditions
  and ${\mathcal X}= {\mathcal X}_n$ the set of $ \btheta=(\theta^1, \ldots, \theta^n)\in {\mathcal X}'$ such that we can find representatives with
\begin{equation}  \label{nyd.2}
\theta^1<\theta^2< \cdots < \theta^n <\theta^{n+1}:= \theta^1 + \pi.
\end{equation}
Let $\mathcal X^*_n$ be the set of $\bz = (z^1,\ldots,z^n)$ with
$z^j = \exp \{2i\theta^j\}$  and $\btheta
\in \mathcal X$.  In other words, $\mathcal X^*_n$ is
the set of $n$-tuples of distinct points on the unit
circle ordered counterclockwise (with a choice of 
a first point).    Note
that $\abs{z^j-z^k}=2 \abs{\sin(\theta^k-\theta^j)}$. We let
\[
\begin{aligned}
\psi( \btheta) &= \sum_{j=1}^n \sum_{k\neq j} \csc^2(\theta^j-\theta^k) = 2 \sum_{1\leq j<k\leq n} \csc^2(\theta^j-\theta^k),
\\ F(\btheta)& =   \ \prod_{1\leq j<k\leq n} |\sin(\theta^k - \theta^j)| 
 = 2^{-n(n-1)/2} \prod_{1\leq j<k\leq n} \abs{z_k-z_j},
  \label{nyd.1}
\\ F_\alpha(\btheta)&=F(\btheta)^\alpha
\\ d( \btheta)&=\min_{1\leq j<k\leq n} \abs{\sin(\theta^{j+1}-\theta^j)}
\\     f_\alpha(\btheta)& = \Integral{\alpha}^{-1} \, F_\alpha(\btheta),  \;\;\;\;  \Integral{\alpha} =
 \int_{\mathcal X} F_\alpha(\btheta) \, d \btheta.
\end{aligned}
\]

Here $d \btheta$ denotes integration with respect
to Lebesgue measure restricted to ${\mathcal X}$.
\begin{remark}
We choose to represent points $z^j$ on the unit circle
as $\exp\{2i\theta^j\}$ (rather than $\exp\{i\theta^j\}$)
because the relation
\[  F_\alpha(\btheta) = 2^{-\alpha n(n-1)/2} \prod_{1\leq j<k\leq n} \abs{z_k-z_j}^\alpha,\]
makes it easy to relate measures on ${\mathcal X}_n$
with measures that arise
in random matrices. (See, for example, Chapter 2 of \cite{Forrester} for the distribution of the eigenvalues of the circular $\beta$-ensemble.)
Note that if $\theta^1,\ldots,\theta^n$ are independent
standard Brownian motions, then $z^1,\ldots,z^n$
are independent driftless Brownian motions on the
circle with variance parameter $4$.

\end{remark}


We will use the following trigonometric identity.

\begin{lemma} \label{lemma:trig_identity}
	If $\btheta\in \mathcal X_n$, 
	\begin{equation}\label{cotsum}
	\sum_{j=1}^n \left( \sum_{k\neq j} \cot(\theta^j -\theta^k) \right)^2 = \psi( \btheta)-\frac{n(n^2-1)}{3}.
	\end{equation}
\end{lemma}

\begin{proof}
	We first note that if $x,y,z$ are distinct points
in $[0,\pi)$, then 
\begin{equation}\label{apr5.1}  
                 \cot(x - y) \, \cot (x-z)
                  + \cot (y-x) \, \cot(y-z)
                    +  \cot(z-x) \, \cot(z-y )
                    =  - 1 
                    \end{equation}
Indeed, without loss of generality, we may assume that $0 = x < y < z$
in which case the lefthand side is
\[      \cot(y-z) \,[\cot y - \cot z]  + \cot y \, \cot z  
 \]
which equals $-1$ using the sum formula
\[  \cot(y-z) = \frac{\cot y \, \cot z + 1}{\cot z
- \cot y}\]  

When we expand the square on the lefthand side of \eqref{cotsum}
we get the sum of two terms,
\begin{equation}  \label{jan2.1}
   \sum_{j=1}^n  \sum_{k \neq j  }  \cot^2 (\theta^j - \theta^k) 
\end{equation}
\begin{equation}  \label{jan2.2}
     \sum_{j=1}^n  \sum_{k=1}^n \sum_{l=1}^n
\, 1\{j \neq k, j \neq l, k \neq l \} 
 \cot(\theta^j - \theta^k) \, \cot (\theta^j - \theta^l) . \end{equation}
   Using $\cot^2 y + 1 = \csc^2 y$,
 we see that \eqref{jan2.1} equals
$     \psi(  \btheta)   - n(n-1)
$. We write \eqref{jan2.2} as $2$ times 
  \[    \sum 
        \left[\cot(\theta^j - \theta^k) \, \cot (\theta^j - \theta^l) 
             +  \cot(\theta^k - \theta^j) \, \cot (\theta^k - \theta^l)  +  \cot(\theta^l - \theta^j) \, \cot (\theta^l - \theta^k)  \right],
\]
where the sum is over all $3$ elements subsets $\{j,k,l\}$
of $\{1,\ldots,n\}$.  Using \eqref{apr5.1}, we see that \eqref{jan2.2} equals
\[      -2 \, \binom{n}{3} = -\frac{n(n-1)(n-2)}{3}  . \]
Therefore, the lefthand side of \eqref{cotsum}
equals
\[   \psi(  \btheta)   - n(n-1)-\frac{n(n-1)(n-2)}{3} 
= \psi(  \btheta)   -\frac{n(n^2-1)}{3}.
\]
\end{proof}

We will let $ \btheta_t=(\theta^1_t, \ldots, \theta^n_t)$ be a standard $n$-dimensional Brownian motion in $\mathcal X^*$ starting at $ \btheta_0 \in \mathcal X$ and stopped at
\[
T=\inf\{ t: \btheta_t\nin \mathcal X \}=\inf \{ t: d(\btheta_t)=0\},
\]
defined on the filtered probability space $(\Omega, \mathcal F_t,  \BMProb)$. 

Differentiation using \eqref{cotsum} shows that
\[
\begin{aligned}
\partial_j F_\alpha( \btheta)&=F_\alpha(\btheta) \sum_{k\neq j}\alpha \cot(\theta^j-\theta^k),
\\
\partial_{jj}F_\alpha( \btheta)&=F_\alpha( \btheta) \left[  \left( \sum_{k\neq j}\alpha \cot(\theta^j-\theta^k) \right)^2 -\alpha \sum_{k\neq j} \csc^2(\theta^j-\theta^k) \right],
\\
\Delta F_\alpha( \btheta) &= F_\alpha( \btheta)\left[ -\frac{\alpha^2n(n^2-1)}{3} + (\alpha^2-\alpha)\psi(\btheta) \right].
\end{aligned}
\]
Hence, if we define
\begin{equation}\label{Mart}
\begin{aligned}
M_{t,\alpha}:&= F_\alpha( \btheta_t) \exp \left \{-\frac{1}{2} \int_0^t \frac{\Delta F_\alpha(\btheta_s)}{F_\alpha( \btheta_s)}ds \right\}
\\ &= F_\alpha( \btheta_t) \exp \left \{ \frac{\alpha^2 n(n^2-1)}{6}t \right \} \, \exp \left\{ \frac{\alpha-\alpha^2}{2} \int_0^t \psi( \btheta_s) ds \right \}, \quad t<T,
\end{aligned}
\end{equation}
then $M_{t,\alpha}$  is a local martingale for $0 \leq t < T$
 satisfying
\[
dM_{t,\alpha}=M_{t,\alpha}\,\sum_{j=1}^n \left( \sum_{k\neq j} \alpha \cot(\theta^j_t-\theta^k_t) \right) d\theta^j_t.
\]
We will write $\mathbb P_\alpha,\E_\alpha$ for the probability measure obtained after tilting $\BMProb $ by $M_{t,\alpha}$
using the Girsanov theorem. Then 
\begin{equation}   \label{eve.2}
d\theta^j_t=\alpha \sum_{k\neq j}\cot(\theta^j_t-\theta^k_t)\,dt + dW^j_t,\;\;\;\; t < T, 
\end{equation}
for independent standard Brownian motions $W^1_t, \ldots, W^n_t$ with respect to $\mathbb P_\alpha$. If $\alpha\geq 1/2$, comparison with the usual Bessel process shows that $\mathbb P_\alpha(T=\infty)=1$. In particular, $M_{t,\alpha}$ is a martingale and $\mathbb P_\alpha \ll  {\BMProb }$ on $\mathcal F_t$ for each $t$. (It is not true that ${\BMProb }\ll \mathbb P_\alpha$ since ${\BMProb } \{ T<t \}>0$.)

This leads to the following definitions.

\begin{definition}

The $n$-radial Bessel process with parameter $\alpha$  is the process satisfying
\eqref{eve.2} where $W_t^1,\ldots, W_t^n$ are
independent Brownian motions.  

\end{definition}

\begin{proposition}  If $\btheta_t$ satisfies \eqref{eve.2} 
and $\tilde \btheta_t = \btheta_{t/n}$, then
$\tilde \btheta_t$ satisfies
\begin{equation}  \label{eve.2.new}
          d\tilde \theta_t = 
\frac \alpha n
\sum_{k\neq j}\cot( \tilde \theta^j_t- \tilde \theta^k_t)\,dt + \frac{1}{\sqrt n}\,
d\tilde W^j_t, \;\;\;\; t < \tilde T, 
\end{equation}
where $\tilde W_t^1,\ldots,\tilde W_t^n$ are independent
Brownian motions and $\tilde T = nT$.

\end{proposition}

We also refer to a process satisfying \eqref{eve.2.new} as
the $n$-radial Bessel process.   If $n=2$, $\tilde \theta_t^1,
\tilde \theta_t^2$ satisfy \eqref{eve.2.new} and
\[    X_t = \tilde \theta_t^2 - \tilde \theta_t^1, \;\;\;\;
     B_t = \frac{1}{\sqrt 2} \, [\tilde W_t^2 - \tilde W_t^1],\]
 then $B_t$ is a standard Brownian motion and $X_t$
 satisfies
\[       dX_t = \alpha \, \cot X_t \, dt + d B_t. \]
This equation is called the radial Bessel equation.

\begin{proposition}
Let $p_{t,\alpha}(\btheta,\btheta')$  denote 
the transition density for the system \eqref{eve.2}.
Then for all $t$ and all $\btheta, \btheta'$,
\begin{equation}  \label{jan24.1}
    p_{t,\alpha}(\btheta,\btheta')   
=  \frac{F_{2\alpha}(\btheta')}
{F_{2\alpha}(\btheta)} \, p_{t,\alpha} \, (\btheta',\btheta).
\end{equation}
\end{proposition}

\begin{proof}
Let $p_t = p_{t,0}$ be the transition density for independent
Brownian motions killed at time $T$.
Fix $t,\btheta,\btheta'$.  Let $\gamma:[0,t]
\rightarrow {\mathcal X}$ be any curve with $\gamma(0) =
\btheta,
\gamma(t) = \btheta'$ and note that the Radon-Nikodym
derivative of $\Prob_\alpha$ with respect to $\BMProb$ evaluated
at $\gamma$ is 
\[  Y(\gamma):= \frac{F_{\alpha}(\btheta ')}{F_{\alpha}(\btheta)}
   \,A_t(\gamma) , \;\;\;\; A_t(\gamma) = e^{a^2n(n-1)t/2} \, \exp
   \left\{\frac{-\alpha^2}{2} \int_0^t \psi(\gamma(s)) \, ds
   \right\}. \]
 If $\gamma^R$ is the reversed path, $\gamma^R(s) = \gamma(t-s)$, then $A_t(\gamma^R) = A_t(\gamma)$ and hence
 \[    Y(\gamma^R) =    \frac{F_{\alpha}(\btheta  )}{F_{\alpha}(\btheta')} \, A_t. \]
Therefore,
\[   \frac{Y(\gamma)}{Y(\gamma^R)} =  \frac{F_{\alpha}(\btheta ')^2}{F_{\alpha}(\btheta)^2} = 
   \frac{F_{2\alpha}(\btheta ')}{F_{2\alpha}(\btheta)}.\]
Since the reference measure $\BMProb$ is time reversible
and the above holds for every path, \eqref{jan24.1} holds. 
   \end{proof}

\begin{proposition}
If $\alpha \geq 1/2$ and
 $\btheta_t$ satisfies \eqref{eve.2}, then with
probability one $T = \infty$.
\end{proposition}

\begin{proof}
This follows by comparison with a usual Bessel process;
we will only sketch the proof.
Suppose $\Prob\{T < \infty\} >0$.  Then there would exist
$j < k$ such that with positive probability $\gamma_T^j
= \gamma_T^k$ but $\gamma_T^{j-1} < \gamma_T^j$ and
$\gamma_t^{k+1} > \gamma_t^k$ (here we are using ``modular
arithmetic'' for the indices $j,k$ in our torus).   If $k = j+1$,
then by comparison to the process
\[     dX_t = \left(\frac{a}{X_t} - r\right) \, dt + dB_t, \]
one can see that this cannot happen.  If $k > j+1$,
we can compare to the Bessel process obtained by removing
the points with indices $j+1$ to $k-1$.
\end{proof}

\begin{proposition} \label{prop:invdensity}
If $\alpha \geq 1/2$, then the invariant density for $\eqref{eve.2}$ is $f_{2\alpha}$.  Moreover, there exists $u>0$
such that for all $\btheta, \btheta'$,
\begin{equation} \label{eq:expconv}
  p_t(\btheta',\btheta) = f_{2\alpha}(\btheta)
        \, \left[1 + O(e^{ -ut})\right].
\end{equation}
\end{proposition}

\begin{proof}
The fact that $f_{2\alpha}$ is invariant follows from
\[  \int p_{t,\alpha}(\btheta,\btheta') \, f_{2\alpha}(\btheta)
 \, d\btheta =
    \int p_{t,\alpha}(\btheta',\btheta)\, \frac{F_{2\alpha}(\btheta')}
      {F_{2\alpha}(\btheta)}  \, f_{2\alpha}(\btheta) \,d\btheta
      \hspace{1in} \]
    \[  \hspace{1in}
         = f_{2\alpha}(\btheta')   \int p_{t,\alpha}(\btheta',\btheta) 
         \,d\btheta = 
  f_{2\alpha}(\btheta').\]
 
Proposition \ref{bounds} below shows that there exist $0<c_1<c_2<\infty$ such that for all $\bx, \by \in \mathcal X$,
\begin{equation}\label{bounds1}
c_1 F_{2\alpha}(\by) \leq p_1(\bx, \by) \leq c_2 F_{2\alpha}(\by).
\end{equation}
The proof of this fact is the subject of \S\ref{subsec:exp_rate_of_conv}. The exponential rate of convergence (\ref{eq:expconv}) then follows by a standard coupling argument (see, for example, \S4 of \cite{Lawler_Minkowski_RealLine}).

 
 \end{proof}

 \begin{proposition} \label{prop:P2a_restated}
 	Suppose $\alpha \geq 1/4$ and
 \[    N_{t,\alpha,2\alpha} =  F_a(\btheta_t) \,
 \exp \left \{ \frac{  \alpha^2 n(n^2-1)}{2}t \right \} \exp \left\{ -\alpha b_{\alpha}\int_0^t \psi( \btheta_s) ds \right \} 
 ,\]
 where $b_\alpha = (3\alpha-1)/2$. Then $N_{t, \alpha, 2\alpha}$ is a $\Prob_\alpha$-martingale, and the measure
 obtained by tilting $\Prob_\alpha$ by $N_{t,\alpha,2\alpha}$ is $\Prob_{2\alpha}$.
\end{proposition}

\begin{proof}

Note that 
\[
\begin{aligned}
M_{t,2\alpha}:&=  F_{2\alpha}( \btheta_t) \exp \left \{ \frac{(2\alpha)^2 n(n^2-1)}{6}t \right \} \exp \left\{ \frac{2\alpha-(2\alpha)^2}{2} \int_0^t \psi( \btheta_s) ds \right \},  \\
& =   M_{t,\alpha} \,  N_{t,\alpha,2\alpha}.
\end{aligned}
\]
Since $M_{t,\alpha},M_{t,2\alpha}$ are both local martingales, we
see that $N_{t,\alpha,2\alpha}$ is a local martingale with respect to $\Prob_\alpha$.
Also, the induced measure by ``tilting first by $M_{t,\alpha}$ and then tilting by $N_{t,\alpha,2\alpha}$'' is the same as tilting by $M_{t,2\alpha}$. 
Since $2\alpha \geq 1/2$, we see that with probability one, $T = \infty$ in the new measure, from which we conclude that it is a martingale.
\end{proof}


We now prove Theorem \ref{exponentialrateofconv1}. 

\begin{proof}[Proof of Theorem \ref{exponentialrateofconv1}]
Using the last two propositions, we see that
\begin{eqnarray*}
\E^\btheta_\alpha\left[ \exp \left\{ -\alpha b_\alpha\int_0^t \psi( \btheta_s) ds \right \} \right]& = & 
 \exp \left \{ \frac{ - \alpha^2 n(n^2-1)}{2}t \right \}
  \, \E^\btheta_a[N_t \, F_{-\alpha}(\btheta_t)]\\
& = &   \exp \left \{ \frac{ - \alpha^2 n(n^2-1)}{2}t \right \}
  \,F_\alpha(\btheta)
\,  \E_{2 \alpha}^\btheta\left [ 
F_{-\alpha}(\btheta_t)\right ]\\
& = & e^{-2\alpha n\beta t} \, F_\alpha(\btheta)  \,\frac{\Integral{3\alpha}}{
\Integral{4\alpha}}\,
   [ 1 + O(e^{-ut})]    .
\end{eqnarray*}
In particular, the last equality follows by applying Proposition \ref{prop:invdensity} for $\alpha=2a$.
\end{proof}

Setting $\alpha=a$, we can write the result as
\[ e^{2 n (n-1)\tilde b t} \, \E^\btheta_a\left[ \exp \left\{ -ab\int_0^t \psi( \btheta_s) ds \right \} \right] =  e^{-2an\hat \beta_n t} \, F_a(\btheta)  \,\frac{\Integral{3a}}{
\Integral{4a}}\,
   [ 1 + O(e^{-ut})]  \]
 where
 \[    \hat \beta_n = \beta - \tilde b (n-1),
 \]
 is the $n$-interior scaling exponent, as in (\ref{eq:n_scalingexp}).

\subsection{Rate of convergence to invariant density}\label{subsec:exp_rate_of_conv}

It remains to verify the bounds (\ref{bounds1}) used in the proof of Proposition \ref{prop:invdensity}. While related results have appeared elsewhere, including \cite{ErdosYauBook}, we have not found Proposition  \ref{bounds} in the literature, and so we provide a full proof here.  This is a sharp pointwise result;
however, unlike
 results coming from random matrices, the constants
 depend on $n$ and we prove no uniform result
 as $n \rightarrow \infty$. 
Our  argument 
uses  the general idea of a ``separation lemma''  (originally \cite{Lawler_cutpoints}).

We consider the $n$-radial Bessel process given
by the system (\ref{eve.2}) with  $\alpha \geq 1/2$.   By Proposition \ref{prop:invdensity} its invariant density is $f_{2\alpha}$, so that for $\bx, \by \in \mathcal X$,
\[
\frac{p_t(\bx, \by)}{F_{2\alpha}(\by)}=\frac{p_t(\by, \bx)}{F_{2\alpha}(\bx)}.
\]
We will prove the following.

\begin{proposition}\label{bounds}
	For every positive integer $n$ and $\alpha \geq 1/2$,
	there exist $0<c_1<c_2<\infty$, such that for all $\mathbf x, \mathbf y\in \mathcal X$, 
	\[
	c_1 F_{2\alpha}(\mathbf y)\leq p_1(\mathbf x, \mathbf y)\leq c_2 F_{2\alpha}(\mathbf y).
	\]
\end{proposition}

For the remainder of this section, we fix $n$ and $\alpha \geq 1/2$
and allow constants to depend on $n,\alpha$. 
 We will let $\BMtrans_t (\bx, \by)$ denote the transition density for independent Brownian motions killed upon leaving
$ \mathcal X$.  If $U \subset \mathcal X$ we will write
$p_t(\x,\y; U)$ or $p_t(\x,\y; \overline U)$  for the density
of the  $n$-radial Bessel process  killled upon leaving $U$; we write
$\tilde p_t(\x,\y; U)$ or $ \tilde p_t(\x,\y; \overline U)$ for
the analogous densities for independent Brownian
motions.  Then we have 
\begin{equation}\label{densities}
p_t(\bx, \by;U)=\BMtrans_t(\bx, \by;U) \; \frac{M_{t,\alpha}}{M_{0,\alpha}}.
\end{equation}  We can use  properties of the density of Brownian motion to conclude analogous properties for $p_t(\bx, \by)$. For example we have the following:
\begin{itemize}
\item  For every open $U$ with $\overline U \subset {\mathcal X}$
and every $t_0$, there exists $C = C(U,t_0)$ such that 
\begin{equation}  \label{covid.3}
          C^{-1} \,\tilde p_t(\x,\y,U) \leq  p_t(\x,\y;U)\leq
C\,\tilde   p_t(\x,\y;U), \;\;\; 0 \leq t \leq t_0.
\end{equation}
\end{itemize}
Indeed, $M_{t,\alpha}/M_{0,\alpha}$ is uniformly bounded
away from $0$ and $\infty$ for $t \leq t_0$ and paths
staying in $\bar U$. 
Another example is   the following easy lemma which we
will use in the succeeding lemma.

\begin{lemma} \label{extralemma}
 Suppose that $0 \leq \theta_0^1 < \theta_0^2
 < \theta_0^n < \theta_0^{n+1}$  where  $\theta_j^{n+1} = \theta_t^1 + \pi$. For every $r > 0$, there exists $q   >0$
such that the following holds.
If $\epsilon \leq 1/(8n)$  and $\theta_0^{j+1} - \theta_0^j \geq r \, \epsilon$
for   $j=1,\ldots,n$,  then with probability at least $q$
the following holds:
\[              \theta_t^{j+1} - \theta_t^j\geq 2\epsilon, 
\;\;\;\;\;  \epsilon^2/2 \leq t \leq \epsilon^2, \;\;\; j=1,\ldots,n,\]
\[     \theta_t^{j+1} - \theta_t^j \geq r\epsilon/2,
\;\;\;\;\; 0\leq t \leq \epsilon, \;\;\; j=1,\ldots,n,\]
\[         |\theta_t^j - \theta_0^j| \leq 4n\epsilon, \;\;\;0\leq t \leq \epsilon^2, \;\;\;  
 j=1,\ldots,n.\]
 \end{lemma}

\begin{proof}
If $\theta_t^j$ were independent Brownian motions, then
scaling shows that the probability of the event is independent
of $\epsilon$ and it is easy to see that it is positive.  Also,
on this event  $M_{t,\alpha}/M_{0,\alpha}$ is uniformly
bounded away from $0$   uniformly in 
$\epsilon$. 
\end{proof}

The next lemma shows that there is a
constant $\delta > 0$ such that from any initial configuration,
with probability at least $\delta$  all the
points are separated by $2\epsilon$  by
time $\epsilon^2$.

\begin{lemma}\label{double_distance}
	There exists $\delta>0$ such that if $0<d(\btheta_0)\leq \epsilon<\delta$, then
	\[
	\mathbb P\{ d(\btheta_{\epsilon^2})\geq 2\epsilon \}\geq \delta.
	\]
	Moreover, if $\tau=\tau_\epsilon=\inf\{ t: d(\btheta_t)=2\epsilon \}$, then for all positive integers $k$, 
	\[
	\mathbb P\{\tau>k\epsilon^2\}\leq(1-\delta)^{k}.
	\]
\end{lemma}

\begin{proof}  The second inequality follows immediately from the
first and the Markov property.  We will prove a slightly
stronger version of the first inequality  result.  Let
\[                     Y_t = \max_{j=1,\ldots,n}
     \max_{0 \leq s \leq t}  |\theta_s^j - \theta^j_0|.\]
Then we will show that there exists
  $\delta_n$ such that
\begin{equation}  \label{covid.2}
    \Prob\{d(\btheta_{\epsilon^2})\geq 2\epsilon;
   Y_{\epsilon^2} \leq \delta_n^{-1} \, \epsilon  \}
   \geq \delta_n.
   \end{equation}
  We have put the explicit $n$ dependence on $\delta_n$
  because  our argument will use induction on $n$.
Without loss of generality we assume that \[
	0 = \theta^1<\theta^2<\cdots <\theta^n<\pi
\]
and 
$   \pi - \theta^n \geq \theta^j - \theta^{j-1}$  for $
 j =2\ldots n.$

 For $n=2$,  
 $\theta_t^2 - \theta_t^1$ is a radial Bessel process which
 for small $\epsilon$ is very closely approximated by a 
 regular Bessel process. Either by using the explicit
 transition density  or by scaling,  we see that
 there exists $c_1 > 0$ such that
 for all $\epsilon \leq 1$, if $\theta^2 - \theta^1 \leq \epsilon$, 
 \[   \Prob\{ \theta_{\epsilon^2}^2 - \theta_{\epsilon^2}^1 \geq
2\epsilon \; ; \;  \theta_{t}^2 - \theta _t^1 \leq4 \epsilon
 \mbox { for } 0 \leq t \leq \epsilon^2  \} \geq c_1. \] 
 Let $A_\epsilon $ denote the event
 \[      A_\epsilon = \{ \theta_{\epsilon^2}^2 - \theta_{\epsilon^2}^1 \geq
2\epsilon \; ; \;  \theta_{t}^2 - \theta _t^1 \leq4 \epsilon
 \mbox { for } 0 \leq t \leq \epsilon^2 \;;\;
  |W_t^1| ,  |W_t^2| \leq  u \, \epsilon, \; 0 \leq t \leq \epsilon^2\},\]
  where $u$ is chosen so
  that
\[     \Prob\left\{\max_{0 \leq t \leq 1} |W_t^j| \geq u\right\} = c_1/4.\]
Then $\Prob(A_\epsilon) \geq c_1/2$. 
  Since
   \[  \theta_t^j =   \alpha  \int_0^t \cot(\theta^j_s - \theta^{3-j}_s)
         \, ds  + W_t^j, \] 
    we see that
    \[   2\alpha \int_0^{\epsilon^2} \cot(\theta^2_s - \theta^1_s)
         \, ds  =
             \theta_\epsilon^2 - \theta_\epsilon^1 - W_{\epsilon^2}^2 + W_{\epsilon^2}^1  \leq (4 + 2u)\, \epsilon.\]
and for $0 \leq t \leq \epsilon^2$, 
\[ |\theta_t^j - \theta_0^j| \leq \alpha  \int_0^t \cot(\theta^j_s - \theta^{3-j}_s)
         \, ds  + |W_t^j|  \leq \left(2 + 2u \right) \epsilon. \]
This establishes \eqref{covid.2} for $n=2$.

We now assume that \eqref{covid.2}  holds for all $j < n$. 
 We claim that it suffices to prove that there
  exists $\delta_n > 0$ such that if
$d(\btheta_0) \leq \epsilon < \delta_n$, then
\begin{equation}  \label{covid.1}
        \Prob\{{\tau_\epsilon} \leq \delta^{-2}_n \, \epsilon^2,
                Y_{\tau_\epsilon} \leq  \delta_n^{-1} \, \epsilon
                \} \geq \delta_n . 
                \end{equation}
     If
   we apply \eqref{covid.1} to $\tau_{\delta_n n/2}$
  we can use Lemma \ref{extralemma} to conclude \eqref{covid.2}  for $n$.  
     
  Let us first assume that $\epsilon \leq \delta_{n-1}$ and that
  there exists $j \in \{1,\ldots,n-1\}$ with
  $\theta^{j+1} - \theta^j \geq \epsilon$.  Consider  independent $j$-radial
  and $(n-j)$-radial Bessel processes
  $(\theta_t^1,\ldots,\theta_t^j)$ and
  $(\theta_t^{j+1},\ldots, \theta_t^n)$.  In other words, 
  remove  the terms of the form
  \[            \pm \cot(\theta_t^m - \theta^k_t), \;\;\;\;   
       k \leq j < j+1 \leq m  \]
       from the drift in the $n$-radial Bessel
       process, so that now particles $\{1, \ldots, j\}$ do not interact with particles $\{j+1, \ldots, n\}$.
   Using the inductive hypothesis, we can find a $\lambda$
   such that with probability at least $\lambda$ we have
   \[                d(\btheta_{\lambda^2 \epsilon^2} ) \geq 2\lambda \, \epsilon,\;\;\;\;
    Y_{\lambda^2 \epsilon^2} \leq \frac{\epsilon}{4}. \]
   This calculation is done with respect to the two independent
   processes but we note that on this event,
   \[                          \theta^{j+1}_t - \theta^j_t
    \geq \frac{\epsilon}{2} , \;\;\;\; 0 \leq t \leq \lambda^2 \,
     \epsilon^2 .\]
    Hence we get a lower bound on the Radon-Nikodym derivative
    between the $n$-radial Bessel process and the two independent
    processes.  
 
Now suppose there is no such separation. 
Let $\sigma_\epsilon = \inf\{t: |\theta^n_t - \theta^1_t|
\geq n\, \epsilon \}$; it is possible that $\sigma_\epsilon = 0$.
If there were no other particles, $ \theta^n_t - \theta^1_t$
would be a radial Bessel process.  The addition of other
particles pushes the first particle more to the left
and the $n$th particle more to the right.  Hence
by comparison, we see that
\[               \Prob\{\sigma_\epsilon \leq n^2\,\epsilon^2 \}
   \geq c_1 \]
and as above we can find $u$ such that
\[      \Prob\left\{\sigma_\epsilon \leq n^2\,\epsilon^2
\; ; \;   \max_{0 \leq t \leq n^2 \epsilon^2} |W_t^j|
  \leq u \epsilon \right\} \geq \frac{c_1}{2}, \]
and hence (with a different value of $u$)
\[      \Prob\left\{\sigma_\epsilon \leq n^2\,\epsilon^2
\; ; \;   Y_{\sigma_\epsilon} \leq u \, \epsilon\right\} \geq \frac{c_1}{2}
, \]
  Note that on this good event there exists 
  at least one $j=1,\ldots,n-1$ with 
  $\theta^{j+1}_{\sigma_\epsilon } -
    \theta^j_{\sigma_\epsilon } \geq \epsilon$.

     \end{proof}

	For $\zeta >0 $, let $V_\zeta = \{\btheta \in \mathcal X:
d(\btheta) \geq 2^{-\zeta}\}$ and let  $\sigma_\zeta$ denote the first time that the process enters $V_\zeta$:
\[
\sigma_\zeta= \inf \left \{ t: d(\btheta_t)\geq 2^{-\zeta}\right \}.
\]
Define
\begin{equation}\label{def_r}
r=r(\delta)=\min\{ k \in \Z: 2^{-k}<\delta\},
\end{equation}
where $\delta$ is as in Lemma \ref{double_distance}.  Note
that $r$ is a fixed constant for the remainder of this proof.

\begin{lemma}\label{Enters_by_1/2}
	There exists $\const>0$ such that for any  $\mathbf x
	\in \mathcal X$, 
\begin{equation}
\mathbb P^{\mathbf x} \{\sigma_r \leq 1/4\}\geq \const.
\end{equation}
\end{lemma}

\begin{proof}
	
	We will construct a sequence of times $\frac{1}{8}\leq t_1 \leq t_2 \leq \cdots \leq \frac{1}{4}$ such that
if 
\[	q_k = \inf_{\mathbf x \in V_k } \mathbb P^{\mathbf x} \{\sigma_r \leq t_k\}, 
	\qquad k \in \nat,
	\]
	then
$
	q:= \inf _{k} q_k >0.
$  Using \eqref{covid.3}, we can see that $q_k > 0$ for each $k$.
To show that $q > 0$, it suffices to show that there is a summable
sequence such that  for all $k$ sufficiently
large, $q_{k+1} \geq q_k \, (1-u_k)$.  We will do this with
$u_k = (1-\delta)^{k^2}$ where $\delta$ is 
as in Lemma \ref{double_distance}.

	For this purpose, denote
	\begin{equation}\label{def:s_k}
	s_k= k^2\;  2^{-2(k+1)},
	\end{equation}
	and let $l\geq r$ be sufficiently large so that
	\[
	\sum_{k=l}^\infty s_k \leq \frac{1}{8}.
	\]
	Define the sequence $\{t_k\}$ by
	\begin{equation}\label{def:t_k}
	t_{k}= \begin{cases} \frac{1}{4}, \quad & k\leq l
	\\ t_{k-1} + s_{k-1}, & k> l.
	\end{cases}
	\end{equation}
	This sequence satisfies $\frac{1}{8} \leq t_1 \leq t_2 \cdots \leq \frac{1}{4}$.
	Applying Lemma \ref{double_distance}, we see that if 
	$d(\mathbf x)\leq 2^{-(k+1)}$, then 
	\[\mathbb P^{\mathbf x} \left \{ \sigma_{k}\leq s_k \right \}\geq 1- u_{k}.\]
	Therefore,
	\[
	\begin{aligned}
	\mathbb P^{\mathbf x} \{ \sigma_r\leq t_{k+1} \}
	&\geq \mathbb P^{\mathbf x}\{ \sigma_{k+1}\leq s_{k+1} \}\; \mathbb P^{\mathbf x} \{\sigma_r \leq t_{k+1}\vert \sigma_{k+1}\leq s_{k+1}  \}
	\\ &\geq (1-u_k) \inf_{\mathbf z \in \partial V_k}\Prob^{\mathbf z} \{\sigma_r \leq t_k\},
	\end{aligned}
	\]
	so that for all $k>r$,
	\[
	q_{k+1}\geq q_k (1- u_k).
	\]
\end{proof}

We are now prepared to prove Proposition \ref{bounds};   we prove the upper and lower bounds separately.

\begin{proof}[Proof of Proposition \ref{bounds}, lower bound]
	
	We let
	\begin{equation}
	\hat \epsilon = \inf\left \{ p_t(\mathbf z, \by): \mathbf z \in \bdry V_{r}, \, \by \in V_r, \, \frac{1}{4}\leq t \leq 1 \right\}.
	\end{equation}
	To see that this is positive, we use \eqref{covid.3} to see
that
\[  \hat \epsilon \geq c \,  \inf\left \{ \tilde p_t(\mathbf z, \by;
  V_{r+1}): \mathbf z \in \bdry V_{r}, \, \by \in V_r, \, \frac{1}{4}\leq t \leq 1 \right\},\]
  and straightforward arguments show that the righthand side
  is positive.

	 Lemma \ref{Enters_by_1/2} implies that for $\mathbf y \in V_r$,  $\mathbf x \in \mathcal X$, and $\frac 12 \leq t \leq 1$, 
	 
	 \[    p_t(\mathbf x, \mathbf y) \geq \Prob^x\left\{
	 \sigma_r \leq \frac 14 \right\} \, \hat \epsilon = q \, \hat \epsilon.\]
	 Also, since $F_{2\alpha}$ is bounded uniformly away
	 from $0$ in $V_r$, we have
	 \[    p_t(\y,\x) = \frac{ F_{2\alpha}(\x)}{F_{2 \alpha}(\y)} 
	  \,  p_t(\x,\y) \geq c \, F_{2\alpha}(\x) .\]
 This assumes $\x \in\mathcal X, \y \in V_r$.  More generally,
 if 
 $\mathbf x, \mathbf y \in \mathcal X$, 
\begin{eqnarray*}
p_1(\x,\y) &  \geq &  \int_{V_r}  p_{1/2} (\x,\z) \,
 p_{1/2} (\z,\y) \, d\z \\
 & = &    {F_{2\alpha}(\y)} \int_{V_r}  p_{1/2} (\x,\z) \,
 p_{1/2} (\y,\z) \, F_{2\alpha}(\z)^{-1}\, d\z \\
 & \geq & c \,   {F_{2\alpha}(\y)}
 \end{eqnarray*}
	\end{proof}

	

In order to prove the upper bound, we will need two more lemmas.

	\begin{lemma}\label{upperboundlemma}
If  $r$  is as  defined in (\ref{def_r}), then
		\begin{equation}\label{p*}
		p^*:=\sup \left \{ p_t(\mathbf z, \mathbf y): \mathbf z\in \partial V_{r+1}, \mathbf y \in V_r, 0\leq t\leq 1  \right \} <\infty.
		\end{equation}
	\end{lemma}
	
	\begin{proof}
	Let
	\[   p^\# =  \sup \left \{ p_t(\mathbf z, \mathbf y;
	V_{r+2}): \mathbf z\in \partial V_{r+1}, \mathbf y \in V_r, 0\leq t\leq 1  \right \}  .\]
	By comparison with Brownian motion using \eqref{covid.3}
	we can see that $p^\# < \infty$.  Also by comparison with
	Brownian motion, we can see that there exists $\rho > 0$
	such that  for all  $\x \in \p V_{r+2}$,
	\[    \Prob^\x\{\btheta[0,1] \cap V_{r+1} = \eset \} \geq 
	  \rho. \]
From this and the strong Markov property, we see that
	\[    p^* \leq  p^\# + (1-\rho) \, p^*.\]
	(The term $p^\#$ on the righthand side corresponds to paths from $\x$ to $\y$ that stay within $V_{r+2}$. The term $(1-\rho) \, p^*$ corresponds to paths from $\x$ to $\y$ that hit $\bdry V_{r+2}$.)
	Therefore,
	\[
	p^*\leq \frac{p^\#}{\rho}<\infty.
	\]

	\end{proof}
	
	\begin{lemma}  There exist $c, \beta< \infty$ such that 
	if $t \geq 1$,  
$k > r$,  $\x \in \mathcal X, \y \in \overline V_{k+1} \setminus V_k$,
then
\begin{equation}  \label{covid.5}
 {  p_{2^{-2k}t} (\x,\y; V_k^c)}
	   \leq c \, 2^{ \beta k} \, (1-\delta)^t \, F_{2\alpha}
	(\y),
	\end{equation}
	where $\delta >0$ is the constant in Lemma \ref{double_distance}.
	\end{lemma}
	
	\begin{proof}  We first consider $t=1$. 
For independent Brownian motions, we have
	\[    \tilde p_{2^{-2k}} (\x,\y) \leq c \, 2^{nk}, \;\;\;  \x \in V_{k+1}, \]
	\[            \tilde p_{t} (\x,\y) \leq c \, 2^{nk}, \;\;\; 0 \leq t \leq
	 2^{-k} , \, \;\;\; \x \in \p V_{k+2}. \]
There exist $\lambda,c'$ such that $F_{2\alpha}(\z)
\geq c'\, 2^{-\lambda k}$  for  $\z \in V_{k+2}$,  Using the first
inequality  and
\eqref{covid.3}  we have  
	\[    p_{2^{-2k}} (\x,\y; V_{k+3}) \leq c \, 2^{(n+\lambda) k} 
 \leq c \, 2^{\beta n} \, F_{2\alpha}(\y)  \]
  \[   p_{t} (\z,\y; V_{k+3} ) \leq c \, 2^{\beta k }
   \, F_{2\alpha}(\y) 
	\;\;\;\;\; t \leq 2^{-2k}, \;\;\; \z \in \p V_{k+2} , \]
 where $\beta = n + 2\lambda$. 
Arguing as in the previous
lemma, we can conclude that
\[      p_{2^{-2k}} (\x,\y )  
 \leq c \, 2^{\beta n} \, F_{2\alpha}(\y)  .\]
This gives the first inequality and the second follows from
Lemma  \ref{double_distance} using
\[     p_{2^{-2k}t} (\x,\y; V_k^c)
	   \leq  \Prob^\x\{\btheta[0,2^{-2k}(t-1)] \subset V_k^c\}
	    \, \sup_{\z \in \mathcal X}  p_{2^{-2k}}(\z,\y).\]
\end{proof}

	\begin{proof}[Proof of Proposition \ref{bounds}, upper bound]
%
%
	
	For each $k\in \mathbb N$, define
	\[
	J_k=\sup_{\shortstack{$\scriptstyle{\by \in V_k}$ \\ $\scriptstyle{ \forall \mathbf x}$ \\$\scriptstyle{t_k\leq t\leq 1} $
	}} 
	\frac{p_t(\mathbf x, \mathbf y)}{F_{2\alpha}(\mathbf y)},
	\]
	where $\tilde t_k= t_k + 1/2$, so that
	\[
	\tilde t_{k}=\begin{cases}
	\frac{3}{4}, & k\leq l
	\\ \tilde t_{k-1}+s_{k-1}, & k> l,
	\end{cases}
	\]
	where $s_k$ and $t_k$ are defined in (\ref{def:s_k}) and (\ref{def:t_k}) above.
Since $F_{2\alpha}(\mathbf y) \geq d(\mathbf y)^{n(n-1)/2}$, Lemma \ref{upperboundlemma} implies that for each $k$, $J_k<\infty$.
	We will show that $J_{k+1}\leq J_k + c_k$ for a summable sequence $\{c_k \}$, which implies that $\lim_{k\to \infty} J_k <\infty$.
	
	To bound $J_{k+1}$, notice that if $\mathbf y \in V_{k+1}$ and $t_{k+1}\leq t \leq 1$, we have the decomposition for arbitrary $\mathbf x$:
	\begin{equation}\label{eq:upperdecomp}
	 \frac{p_t(\mathbf x, \mathbf y)}{F_{2\alpha} (\mathbf y)} 
	= \frac{p_t(\mathbf y, \mathbf x)}{F_{2\alpha} (\mathbf x)} 
	=  \frac{p_t(\mathbf y, \mathbf x; \sigma_k \leq s_k)}{F_{2\alpha} (\mathbf x)} 
	+ \frac{p_t(\mathbf y, \mathbf x; \sigma_k > s_k)}{F_{2\alpha} (\mathbf x)} .
	\end{equation}
	The strong Markov property implies that the first term on the righthand side is equal to
	\[
	\int_{\partial V_k}p_{\sigma_k} (\mathbf y, \mathbf z;\sigma_k\leq s_k) \frac{p_{t-\sigma_k} (\mathbf z, \mathbf x)}{F_{2\alpha}(\mathbf x)} \, d\mathbf z
	\, \,\, \leq \,\mathbb P^{\mathbf y} \{ \sigma_k\leq s_k \} \sup_{\shortstack{ $\scriptstyle{ \mathbf z \in \partial V_k }$ \\ $ \scriptstyle{\tilde t_{k} \leq s \leq 1}$  
	} } \frac{p_s(\mathbf x', \mathbf z)}{F_{2\alpha} (\mathbf z)}
	\,\,\leq \,\,J_k.
	\]
Using \eqref{covid.5}. we can see  that the second term on the righthand side of (\ref{eq:upperdecomp}) may be bounded by
	\[
	\begin{aligned}
	 \frac{p_t(\mathbf y, \mathbf x; \sigma_k >s_k)}{F_{2\alpha} (\mathbf x)}
	 & =  \int_{\mathcal X_n} p_{s_k} (\mathbf y, \mathbf z; \sigma_k >s_k) \; \frac{p_{t-s_k} (\mathbf z, \mathbf x)}{F_{2\alpha}(\mathbf x)}\,d\mathbf z \\
	 & =  \int_{V_k^c} p_{s_k} (\mathbf y, \mathbf z; V_k^c)\; \frac{p_{t-s_k} (\mathbf x, \mathbf z)}{F_{2\alpha}(\mathbf z)}\,d\mathbf z
	\\ &  = \frac{1}{F_{2\alpha}(\y)}
	 \int_{V_k^c} p_{s_k} (\mathbf z, \mathbf y; V_k^c)\;  {p_{t-s_k} (\mathbf x, \mathbf z)} \,d\mathbf z
	\\ 
	& \leq    e^{-O(k^2)}  \int_{V_k^c}   {p_{t-s_k} (\mathbf x, \mathbf z)} \,d\mathbf z \leq  e^{-O(k^2)} .
	\end{aligned}
	\]
Therefore,
	\[
	J_{k+1} \leq J_k + e^{-O(k^2)},
	\]
	completing the proof.

\end{proof}
\bibliography{bibfile}
\bibliographystyle{alpha}

 \end{document}